\documentclass[11pt,reqno]{amsart} %  please use amsart at 11pt
\usepackage{amsfonts}
\usepackage{mathrsfs}
\usepackage{amssymb,latexsym}
\usepackage[draft=true]{hyperref}
\usepackage{cite} % to get Refs. [1,2,3] typeset as [1--3] automatically
\usepackage{rotating} % rotation
\usepackage[height=190mm,width=130mm]{geometry} % this is the journal
\usepackage{stackengine}
\setstackgap{S}{1pt}

\theoremstyle{plain}
\newtheorem{theorem}{Theorem}

\newtheorem{corollary}{Corollary}

\theoremstyle{definition}
\newtheorem{definition}{Definition}

\theoremstyle{remark}
\newtheorem{remark}{Remark}
\newtheorem{example}{Example}

\numberwithin{equation}{section} % to get equations numbered
\allowdisplaybreaks
\def\mbf#1{\mbox{\boldmath$#1$}}

\begin{document}

\title[\hspace{-6mm} Matrix Differential Operator Method for Particular Solution ODE]
{Matrix Differential Operator Method\\ of Finding a Particular Solution\\
to a Nonhomogeneous Linear Ordinary Differential Equation\\ with Constant Coeficients}
  
\author{Jozef Fecenko}
\address{University of Economics\\ Faculty of Economic Informatics\\ Department of Mathematics and Actuarial Science\\  Dolnozemsk\'a St.\\ 832 04 Bratislava\\ Slovakia}
\email{jozef.fecenko@euba.sk}

\thanks{This paper was supported by the Slovak Grant Agency VEGA No. 1/0647/19. I owe thanks to
J\' an Bu\v sa and Graham Luffrum for their valuable comments and advice that helped improve the text in many places.}

\begin{abstract}
The article presents a matrix differential operator and a pseudoinverse matrix
differential operator for finding a particular solution to nonhomogeneous linear ordinary differential equations (ODE) with constant coeficients with special types of the right-hand side. Calculation requires the determination of an inverse or pseudoinverse matrix. If the matrix is singular, the Moore-Penrose pseudoinverse matrix is used for the calculation, which is simply calculated as the inverse submatrix of the considered matrix. It is shown that block matrices are effectively used to calculate a particular solution. 
\end{abstract}



\maketitle

\section{Introduction}

The idea of representing the processes of calculus, differentiation and integration, as operators has a long history that goes back to the prominent German polymath Leibniz, G.W. The French mathematician L. F. A. Arbogast  was the first mathematician to separate the symbols of operation from those of quantity, introducing systematically the operator notation DF for the derivative of the function. This approach was further developed by F. J. Servois who developed convenient notations. He was followed by a school of British and Irish mathematicians. R. B. Carmichael and  G. Boole describing the application of operator methods to ordinary and partial differential equations. Other prominent personalities who contributed to the development of operational calculus were, for example, physicist O. Heaviside, mathematicians T.J. Bromwich, J.R. Carson, J. Mikusi{\'n}ski, N. Wiener, and others \cite{W}. 

Although the problem of solving ordinary nonhomogeneous linear differential equations with constant coefficients is generally known, we will deal with the method of finding a particular solution of differential equation using a matrix differential operator method. This method, like the method of undetermined coefficients, can be used only in special cases, if the right-hand side of the differential equation is typical, i.e. it is a constant, a polynomial function, exponential function $e^{\alpha x}$, sine or cosine functions $\sin{\beta x}$ or $\cos {\beta x}$, or finite sums and products of these functions ($\alpha$, $\beta$ constants). In some cases, the method can be used as a support method for determining the particular solution of a differential equation using the method of undetermined coefficients or the differential operator method or for evaluation of indefinite integrals. 

We introduce the differential operator notation.

It is sometimes convenient to adopt the notation
$Dy$, $D^{2}y$, $D^{3}y,\cdots, D^{n}y$ to denote $\dfrac{dy}{dx}$,  $ \dfrac{d^{2}y}{dx^{2}}$, $\dfrac{d^{3}y}{dx^{3}},\cdots,\dfrac{d^{n}y}{dx^{n}}$. The symbols $Dy$, $D^{2}y,\ldots$ are called {\it{differential operators}}
\cite{Ch} and have properties analogous to those of algebraic quantities \cite{H}. \\

\begin{tabular}{l}
Table 1:  Operator Techniques \cite{Ch}, \cite{S}\\
\end{tabular}\\
\begin{tabular}{|l|l|}
\hline
\textbf{A.} & \\
\qquad $\phi(D)\sum\limits_{k=0}^nc_kR_k(x)$& 
$\sum\limits_{k=0}^nc_k\phi(D)R_k(x)$\\
\hline
\textbf{B.} & \\
\qquad $\phi(D) \cdot \psi(D)$ & $\psi(D) \cdot \phi(D)$ \\
& $\phi(D),\,  \psi(D)$ {\small polynomial operators in $D$}\\
\hline
\textbf{C.} & \\
\qquad $\phi(D)(xf(x))$ &  $x\phi(D)f(x) + \phi^\prime(D)f(x)$\\
\hline
\textbf{D.}  & \\
\qquad$D^ne^{\alpha x}$& $\alpha^ne^{\alpha x}$,  {\small $n$ is non-negative integer}\\
\hline
\textbf{E.}  & \\
\qquad$D^{n}\left\lbrack e^{ax}R(x) \right\rbrack$ &
$e^{ax}\left( D + \alpha \right)^{n}R(x)$\\
&{\small $n$ is non-negative integer}\\
\hline
\textbf{F.}&\\
\qquad$D^{2n}\sin{\beta x}$&$\left(-\beta^2\right)^n\sin{\beta x}$\\
\qquad $D^{2n}\cos{\beta x}$&$\left(-\beta^2\right)^n\cos{\beta x}$\\
\hline 
\end{tabular}
\par
\vspace{3mm}
\noindent Using the operator notation, we shall agree to write the differential equation
\begin{equation}\label{DR} %%(1,1)
a_{n}y^{(n)} + a_{n - 1}y^{(n - 1)} + \ldots + a_{1}y^{\prime} + a_{0}y = f(x),\quad a_n \neq 0,\ a_i\in \mathbb{R}
\end{equation}
as
\begin{equation}\label{DRD}
\left(a_{n}D^{n} + a_{n - 1}D^{n - 1} + \cdots + a_{1}D + a_{0}\right)y = f(x)
\end{equation}
or briefly
\begin{equation}\label{DRDp}
\phi(D)y = f(x),
\end{equation}
where
\begin{equation}\label{DO}
\phi(D) = a_nD^n + a_{n - 1}D^{n - 1} + \cdots + a_{1}D + a_{0}
\end{equation}
is called an \emph{operator polynomial} in \(D.\)
If we want to emphasize the degree of the polynomial operator \eqref{DO}, we shall
write it in the form \(\phi_{n}( D )\).
\vspace{5mm}

The use of the operator calculus for solving linear differential equations is well dealt with in several publications, e.g. \cite{H}, \cite{Ch}, \cite{S}. 

\section{Matrix differential operator} 
\begin{definition}
If \(S = \{\mbf{v}_{1},\mbf{v}_{2},\ldots,\mbf{v}_{n}\}\)
is a set of vectors in a vector space \(V,\) then the set of all linear
combination of
\(\mbf{v}_{1},\mbf{v}_{2}, \ldots,\mbf{v}_{n}\) is called the \emph{span} of
\(\mbf{v}_{1},\mbf{v}_{2},\ldots,\mbf{v}_{n}\) and is denoted by
\(\text{span}(\mbf{v}_{1},\mbf{v}_{2},\ldots,\mbf{v}_{n})\)
or \text{span}(S).
\end{definition}

Let \(G\) be a~vector space of all differentiable function. Consider the
subspace \(V \subset G\) given by
\begin{equation}\label{spanfi}
V=\text{span}\left( f_{1}(x),f_{2}(x),\cdots,f_{n}(x) \right),
\end{equation}
where we assume that the functions \(f_{1}(x),f_{2}(x),\cdots,f_{n}(x)\)
are linearly independent. Since the set 
\(B = \{ f_{1}(x),f_{2}(x),\cdots,f_{n}(x)\}\) is linearly independent, it is a~basis for V.\newline
The functions \(f_{i}(x),\ i = 1,2,\ldots,n\) expressed in basis \(\mbf{B}\) using base vector coordinates are usually written
\[\left\lbrack f_{1}(x)\right\rbrack_{B} = \begin{bmatrix}
1 \\
0 \\
\vdots \\
0 \\
\end{bmatrix},\ \left\lbrack f_{2}(x)\right\rbrack_{B} = \begin{bmatrix}
0 \\
1 \\
\vdots \\
0 \\
\end{bmatrix},\cdots,\left\lbrack f_{n}(x) \right\rbrack_{B} = \begin{bmatrix}
0 \\
0 \\
 \vdots \\
1 \\
\end{bmatrix}\]
The vector \(\left\lbrack f_{i}(x) \right\rbrack_{B}\) has in the \emph{i-}th row 1 and~0 otherwise.

Further, assume that the differential operator \( D\) maps \(V\)  into itself\\
 Let 
\[D( f_{i}( x) )= \sum_{j = 1}^{n}c_{ij}f_{j}( x),\quad  i = 1,2,\cdots,n,\]
where \(c_{ij} \in \mathbb{R}\),  \(i,j = 1,2,\ldots,n\) are constants.
Then
\[\left\lbrack D\left( f_{i}(x) \right) \right\rbrack_{B} = \begin{bmatrix}
c_{i1} \\
c_{i2} \\
 \vdots \\
c_{in} \\
\end{bmatrix},\quad  i,j = 1,2,\ldots,n \]
and (see \cite{P})
\begin{equation}\label{Poole}
\left\lbrack D \right\rbrack_{B} = \left\lbrack \left\lbrack D\left( f_{1}(x) \right) \right\rbrack_{B} \Shortstack{. . . .}\left\lbrack D\left( f_{2}(x) \right) \right\rbrack_{B} \Shortstack{. . . .}\cdots  \Shortstack{. . . .}\left\lbrack D\left( f_{n}(x) \right) \right\rbrack_{B} \right\rbrack = \begin{bmatrix}
c_{11} & c_{21} & \cdots & c_{n1} \\
c_{12} & c_{22} & \cdots & c_{n2} \\
 \vdots & \vdots & \ddots & \vdots \\
c_{1n} & c_{2n} & \cdots & c_{nn} \\
\end{bmatrix}
\end{equation}
If
\[f(x) = \sum_{i = 1}^{n}{\alpha_{i}f_{i}(x),\quad \alpha_{i} \in \mathbb{R},\quad i = 1,2,\cdots,n},\quad  f(x) \in V \]
then
\[\left\lbrack f(x) \right\rbrack_{B} = \begin{bmatrix}
\alpha_{1} \\
\alpha_{2} \\
 \vdots \\
\alpha_{n} \\
\end{bmatrix}\]
We express what is the derivative of the function \(f(x)\)
\[Df(x) = \sum_{i = 1}^{n}{\alpha_{i}Df_{i}(x)} = \sum_{i = 1}^{n}\alpha_{i}\sum_{j = 1}^{n}{c_{ij}f_{j}(x)} = \sum_{j = 1}^{n}{\sum_{i = 1}^{n}\alpha_{i}c_{ij}f_{j}(x)}\]
respectively
\[\left\lbrack D( f(x)) \right\rbrack_{B}= \begin{bmatrix}
\sum\limits_{i=1}^n c_{i1}\alpha_{i} \\
\sum\limits_{i=1}^n c_{i2}\alpha_{i} \\
 \vdots \\
\sum\limits_{i=1}^n c_{in}\alpha_{i} \\
\end{bmatrix}=\begin{bmatrix}
c_{11} & c_{21} & \cdots & c_{n1} \\
c_{12} & c_{22} & \cdots & c_{n2} \\
 \vdots & \vdots & \ddots & \vdots \\
c_{1n} & c_{2n} & \cdots & c_{nn} \\
\end{bmatrix}\begin{bmatrix}
\alpha_{1} \\
\alpha_{2} \\
 \vdots \\
\alpha_{n} \\
\end{bmatrix}\]
Let us next simply denote \(\left\lbrack D \right\rbrack_{B}\) as  \(\mathcal{D}_B\). 

The matrix \(\mathcal{D}_B\) we will called a \emph{matrix differential operator  corresponding to a vector space} \(V\)  \emph{with the considered basis} \(B\). 

Denote  
\[\left\lbrack\mathcal{D}(f(x) \right\rbrack_{B}=\mbf{f}_B^\prime = \begin{bmatrix}
\beta_{1} \\
\beta_{2} \\
\mbf{\vdots} \\
\beta_{n} \\
\end{bmatrix}\,\,\text{and\ \ }\left\lbrack f(x) \right\rbrack_{B} = \mbf{f}_B\]
then
\[\mbf{f}_B^\prime=\mathcal{D}_B\mbf{f}_B\]

Note that the matrix transformation \(\mathcal{D}_B:V \rightarrow V\) defined by
\[\mathcal{D}_B\mbf{f}_B = \mbf{f}_B^\prime\]
is a linear transformation.

As mentioned in \eqref{Poole}, \emph{i-}th, \(\left( i = 1,2,\cdots,n \right) \) column of the matrix \(\mathcal{D}\) expresses the derivative of the function \(f_{i}(x).\)

The following properties of the matrix differential operator are given without proof.
\begin{theorem}\label{1} Let \(\mathcal{D}_B\) be a matrix differential
operator of vector space \(V\) with a basis \(B\). Then
\begin{itemize}
\item
\(\mathcal{D}_B\left(\mbf{f}_B+\mbf{g}_B\right)=\mathcal{D}_B\mbf{f}_B+\mathcal{D}_B\mbf{g}_B\)
\item
  \(\mathcal{D}_B\left( c\mbf{f}_B \right)=c\mathcal{D}_B\left(\mbf{f}_B  \right), \quad c\in\mathbb{R}\)
\item
  \(\mbf{f}_B^{\prime\prime}= \mathcal{D}_B\left(\mathcal{D}_B\mbf{f}_B\right)= \mathcal{D}_B^{2}\mbf{f}_B\)
\item
  \(\mbf{f}_B^{(n)}=\mathcal{D}_B\mbf{f}_B^{( n - 1) }, \quad
\mbf{f}_B^{( 0 )}=\mbf{f}_B\)
\item
  \(\mbf{f}_B^{( n )}=\mathcal{D}_B^{{n}}\mbf{f}_B\)
\item
  \emph{i-}th \(\left( i = 1,2,\cdots,n \right)\ \)column of the matrix
  \(\mathcal{D}_B^{n}\) expresses \it{n}-th derivative of function \(f_{i}( x).\)
\end{itemize}
\end{theorem}

\begin{example}\label{exam1}
 Let us consider
\begin{equation}\label{eq1}
V = \text{span}(xe^{2x}\sin{3x}, xe^{2x}\cos{3x},e^{2x}\sin{3x},e^{2x}\cos{3x})
\end{equation}
The Wronskian
\(W(xe^{2x}\sin{3x},  xe^{2x}\cos{3x},e^{2x}\sin{3x},e^{2x}\cos{3x}) = 324 \cdot e^{8x} \neq 0.\)
It follows that the functions in \eqref{eq1} are linearly independent. Then the set
\(B = \{xe^{2x}\sin{3x},  xe^{2x}\cos{3x},e^{2x}\sin{3x},e^{2x}\cos{3x} \}\)
is a basis for V.

Applying differential operator \emph{D} to a general element of
\emph{V}, we see that
\[\aligned
D&\left({ax}e^{2x}\sin{3x} + bxe^{2x}\cos{3x} +ce^{2x}\sin{3x} + de^{2x}\cos{3x} \right) \\
&=( 2a - 3b )xe^{2x}\sin{3x} + ( 3a + 2b)xe^{2x}\cos{3x}+(a + 2c - 3d)e^{2x}\sin{3x} \\
&+(b + 3c + 2d)e^{2x}\cos{3x},
\endaligned\]
which is again in \emph{V.}

We found the corresponding matrix differential operator \(\mathcal{D}_B\)
in the vector space \(V\) with the basis \(B\). Here the construction of
\(\mathcal{D}_B\) is shown schematically
\end{example}

 \parbox{\textwidth}{
 \hspace{6.6cm}\rotatebox{90}{\small $D\left(xe^{2x}\sin{3x}\right)$}
\hspace{0.05cm}
\rotatebox{90}{\small $D\left(xe^{2x}\cos{3x}\right)$}
\hspace{0.03cm}
\rotatebox{90}{\small $D\left(e^{2x}\sin{3x}\right)$}
\hspace{0.01cm}
\rotatebox{90}{\small$D\left(e^{2x}\cos{3x}\right)$}
\vspace{-0.1cm}
\begin{equation}\label{schemat}
\mathcal{D}_B=\begin{matrix}
xe^{2x}\sin{3x}\\
xe^{2x}\cos{3x}\\
e^{2x}\sin{3x}\\
e^{2x}\cos{3x}\\
\end{matrix}
\begin{bmatrix}
2 & -3 & 0 & 0\\
3 &  2 & 0 & 0\\
1 & 0  & 2 & -3\\
0 & 1  & 3 & 2\\
\end{bmatrix}
\end{equation}
  }%

Let us calculate the first and second derivative of the function
\[f(x) = xe^{2x}\sin{3x} + 2xe^{2x}\cos{3x} - 2e^{2x}\cos{3x} \]
using the matrix differential operator \eqref{schemat}.
We have
\[\left\lbrack\mathcal{D}f(x)\right\rbrack_{B}=\mathcal{D}_{B}\left\lbrack f(x) \right\rbrack_{B} = \begin{bmatrix}
2 & - 3 & 0 & 0 \\
3 & 2 & 0 & 0 \\
1 & 0\  & 2 & - 3 \\
0 & 1 & 3 & 2 \\
\end{bmatrix}\begin{bmatrix}
1 \\
2 \\
0 \\
 - 2 \\
\end{bmatrix} = \begin{bmatrix}
 - 4 \\
7 \\
7 \\
 - 2 \\
\end{bmatrix}\]

So that
\[Df(x) = - 4xe^{2x}\sin{3x} + 7xe^{2x}\cos{3x} + 7e^{2x}\sin{3x} - 2e^{2x}\cos{3x}\]
The second derivative
\[\mbf{f}_B^{\prime\prime}=\mathcal{D}_B\mbf{f}_B^\prime =\begin{bmatrix}
2 & - 3 & 0 & 0 \\
3 & 2 & 0 & 0 \\
1 & 0\  & 2 & - 3 \\
0 & 1 & 3 & 2 \\
\end{bmatrix}\begin{bmatrix}
 - 4 \\
7 \\
7 \\
 - 2 \\
\end{bmatrix} = \begin{bmatrix}
 - 29 \\
2 \\
16 \\
24 \\
\end{bmatrix}\]

\[D^{2}f(x) = - 29xe^{2x}\sin{3x} + 2xe^{2x}\cos{3x} + 16e^{2x}\sin{3x} + 24e^{2x}\cos{3x}\]

For calculating higher powers of the matrix \(\mathcal{D}_B\) it is sometimes advantageous to express it as a block matrix (see example 11) or use Jordan matrix decomposition.

The point of a matrix differential operator is not that  this method is easier than a direct differentiation. Indeed, once the matrix $\mathcal{D}_B$ has been established, it is easy to find the differentials with little to do. What is significant is that matrix methods can be used at all in what appears, on the surface, to be a~calculus problem.
\pagebreak
\section{Some properties of block matrices and pseudoinverse matrix}

\begin{theorem}\label{T2} \cite{Kl}
Let \(\mbf{A}\) be a~matrix partitioned into four $2\times 2$ blocks

\[\mbf{A} = \begin{bmatrix}
\mbf{P} & \mbf{Q} \\
\mbf{R} & \mbf{S} \\
\end{bmatrix}\]
where \(\mbf{R}\) is a regular matrix of the order $r$. Then the rank of matrix \(\mbf{A}\) is equal to $r$ if and only if 
\[\mbf{Q} = \mbf{P}\mbf{R}^{- 1}\mbf{S}.\]
\end{theorem}
\begin{proof}
To the first  block row of the matrix \(\mbf{A}\) we add
the second block row multiplied from the left by the matrix
\(- \mbf{P}\mbf{R}^{- 1}\). This adjustment adds the linear
combination of matrix rows \(( \mbf{R\quad  S})\) to
the other matrix rows. After this adjustment, we get the matrix
\begin{equation}\label{rank}
\begin{bmatrix}
\mbf{0} & \mbf{Q - P}\mbf{R}^{- 1}\mbf{S} \\
\mbf{R} & \mbf{S} \\
\end{bmatrix}.
\end{equation}
For matrix \(\mbf{A}\) and matrix \eqref{rank} to have the same rank, the
matrix \(\mbf{Q - P}\mbf{R}^{- 1}\mbf{S}\) must be equals
\(\mbf{0}.\) It follows that
\(\mbf{Q = P}\mbf{R}^{- 1}\mbf{S}.\)
\end{proof}

\begin{definition} Let \(\mbf{A}\) be a \(m \times n\) matrix. The matrix \(\mbf{X}\) for which
\[\mbf{AXA} = \mbf{A},\]
is called the \emph{pseudoinverse matrix} of matrix A. The pseudoinverse
matrix of a matrix \(\mbf{A}\) is denoted \(\mbf{A}^{-}.\)
\end{definition}

\begin{theorem}\label{T3} 
Let \(\mbf{A}\) be a \(m \times n\) matrix partitioned into four $2\times 2$ blocks
\[\mbf{A} = \begin{bmatrix}
\mbf{P} & \mbf{Q} \\
\mbf{R} & \mbf{S} \\
\end{bmatrix}\]
and let~the rank of the matrix \(\mbf{A}\) be $r$. If the matrix \(\mbf{R}\) is regular
of order $r$, then the \(n \times m\) matrix
\[\mbf{A}^{-} = \begin{bmatrix}
\mbf{0} & \mbf{R}^{- 1} \\
\mbf{0} & \mbf{0} \\
\end{bmatrix}\]
 is the pseudoinverse of the matrix \(\mbf{A}.\)
\end{theorem}

\begin{proof} Compute 
\[\mbf{A}\mbf{A}^{-}\mbf{A} = \begin{bmatrix}
\mbf{P} & \mbf{Q} \\
\mbf{R} & \mbf{S} \\
\end{bmatrix}\begin{bmatrix}
\mbf{0} &\mbf{R}^{- 1} \\
 \mbf{0} & \mbf{0} \\
\end{bmatrix}\begin{bmatrix}
\mbf{P} & \mbf{Q} \\
\mbf{R} & \mbf{S} \\
\end{bmatrix} = \begin{bmatrix}
\mbf{P} & \mbf{P}\mbf{R}^{- 1}\mbf{S} \\
\mbf{R} & \mbf{S} \\
\end{bmatrix} = \begin{bmatrix}
\mbf{P} & \mbf{Q} \\
\mbf{R} & \mbf{S} \\
\end{bmatrix} = \mbf{A}\]
We have used the consequence of the Theorem \ref{T2}, i.e.
\(\mbf{Q} =\mbf{P}\mbf{R}^{- 1}\mbf{S}.\)
\end{proof}

\begin{definition}
 The pseudoinverse matrix \(\mbf{A}^{-}\) of the matrix \(\mbf{A}\) satisfying all of the following four conditions
\begin{itemize}
 \item[(i)] 
\(\mbf{A}\mbf{A}^{-}\mbf{A} = \mbf{A}\),
  (\(\mbf{A}^{-}\) is the pseudoinverse matrix of the matrix
  \(\mbf{A})\);
 \item[(ii)] 
 \(\mbf{A}^{-}\mbf{A}\mbf{A}^{-} = \mbf{A}^{-}\),
 (\(\mbf{A}\) is the pseudoinverse matrix of the matrix
  \(\mbf{A}^{-})\);
\item[(iii)] 
 \(\left( \mbf{A}\mbf{A}^{-} \right)^{T}=\mbf{A}\mbf{A}^{-},\) (\(\mbf{A}\mbf{A}^{-}\)   is a symmetric matrix, \(\left( \cdot \right)^{T}\) means a transposed  matrix);
\item[(iv)] 
 \(\left( \mbf{A}^{-}\mbf{A} \right)^{T}=\mbf{A}^{-}\mbf{A}\),
  (\(\mbf{A}^{-}\mbf{A}\) is a symmetric matrix)
\end{itemize}
is called the Moore-Penrose pseudoinverse of the matrix \(\mbf{A}\),
denoted by \(\mbf{A}^{+}.\)
\end{definition}

\begin{theorem}\label{T33}
Let \(\mbf{A}\) be a \(m \times n\) matrix
partitioned into four $2\times 2$ blocks
\[\mbf{A} = \begin{bmatrix}
\mbf{0} & \mbf{0} \\
\mbf{R} & \mbf{0} \\
\end{bmatrix}\]
and let the rank of matrix \(\mbf{A}\) be \(r.\) If matrix
\(\mbf{R}\) is regular of order \(r\), then
\[\mbf{A}^{-} = \begin{bmatrix}
\mbf{0} & \mbf{R}^{- 1} \\
\mbf{0} & \mbf{0} \\
\end{bmatrix}\]
\(n \times m\) is the Moore-Penrose pseudoinverse matrix of the matrix
\(\mbf{A}\), thus \linebreak \(\mbf{A}^{-}=\mbf{A}^{+}.\)
\end{theorem}
\begin{proof} We will verify all four conditions of definition 2.\vspace{2mm}
\begin{itemize}
\item[(i)]
  \(\mbf{A}\mbf{A}^{-}\mbf{A =}\begin{bmatrix} \mbf{0} & \mbf{0} \\ \mbf{R} & \mbf{0} \\ \end{bmatrix}\begin{bmatrix} \mbf{0} & \mbf{R}^{- 1} \\ \mbf{0} & \mbf{0} \\ \end{bmatrix}\begin{bmatrix} \mbf{0} & \mbf{0} \\ \mbf{R} & \mbf{0} \\ \end{bmatrix} = \begin{bmatrix} \mbf{0} & \mbf{0} \\ \mbf{R} & \mbf{0} \\ \end{bmatrix} = \mbf{A}\)\\
The condition \emph{i} is met.\vspace{2mm}
\item[(ii)]
Pseudoinverse matrix of the matrix 
\[\mbf{A}^{-} = \begin{bmatrix}
\mbf{0} & \mbf{R}^{- 1} \\
\mbf{0} & \mbf{0} \\
\end{bmatrix}\]
is the matrix
\[\mbf{A} = \begin{bmatrix}
\mbf{0} & \mbf{0} \\
\mbf{R} & \mbf{0} \\
\end{bmatrix},\]
because
\[\mbf{A}^{-}\mbf{A}\mbf{A}^{-}=\begin{bmatrix}
\mbf{0} & \mbf{R}^{- 1} \\
\mbf{0} & \mbf{0} \\
\end{bmatrix}\begin{bmatrix}
\mbf{0} & \mbf{0} \\
\mbf{R} & \mbf{0} \\
\end{bmatrix}\begin{bmatrix}
\mbf{0} & \mbf{R}^{- 1} \\
\mbf{0} & \mbf{0} \\
\end{bmatrix} = \begin{bmatrix}
\mbf{0} & \mbf{R}^{- 1} \\
\mbf{0} & \mbf{0} \\
\end{bmatrix} = \mbf{A}^{-}\]
We have also proved the validity of condition \emph{ii}.
\item[(iii)]
  \(\left( \mbf{A}\mbf{A}^{-} \right)^{T}=\left\lbrack \begin{bmatrix} 
\mbf{0} & \mbf{0} \\ \mbf{R} & \mbf{0} \\ \end{bmatrix}\begin{bmatrix} \mbf{0} & \mbf{R}^{- 1} \\ \mbf{0} & \mbf{0} \\ \end{bmatrix} \right\rbrack^{T} = \begin{bmatrix} \mbf{0} & \mbf{0} \\ \mbf{0} & \mbf{I} \\ \end{bmatrix}^{T} = \begin{bmatrix} \mbf{0} & \mbf{0} \\ \mbf{0} & \mbf{I} \\ \end{bmatrix}\)\\
We have shown that condition \emph{iii} is met.\vspace{2mm}
\item[(iv)]
  \(\left( \mbf{A}^{-}\mbf{A} \right)^{T}=\left( \begin{bmatrix} \mbf{0} & \mbf{R}^{- 1} \\ \mbf{0} & \mbf{0} \\ \end{bmatrix}\begin{bmatrix} \mbf{0} & \mbf{0} \\ \mbf{R} & \mbf{0} \\ \end{bmatrix} \right)^{T} = \begin{bmatrix} \mbf{I} & \mbf{0} \\ \mbf{0} & \mbf{0} \\ \end{bmatrix}^{T} = \begin{bmatrix} \mbf{I} & \mbf{0} \\ \mbf{0} & \mbf{0} \\ \end{bmatrix}\)\\
We have shown that condition \emph{iv} is also met.
\end{itemize}
\end{proof}

\begin{corollary}
For the matrices \(\mbf{A}\) and \(\mbf{A}^{-}\) in Theorem \ref{T33} imply that
\(\mbf{A}\mbf{A}^{-}\) and \(\mbf{A}^{-}\mbf{A}\) are incomplete identity matrices.
\end{corollary}

\begin{theorem}\label{SLR} Let
\begin{equation}\label{slr}
\mbf{A}\mbf{x}=\mbf{b}
\end{equation}
be a solvable system of linear equations and \(\mbf{A}^{-}\) be a pseudoinverse of the matrix  \(\mbf{A},\) then
\begin{equation}\label{rslr}
\mbf{x} = \mbf{A}^{-}\mbf{b}
\end{equation}
is a solution of \eqref{slr}.
\end{theorem}
\begin{proof} Assuming that the system of linear equations \eqref{slr} is
solvable, then a vector \(\mbf{y}\)  exists, such that
\[\mbf{A}\mbf{y}=\mbf{b}\]
We show that the vector in \eqref{rslr}  is the solution of  the system of linear equations \eqref{slr}.
It is valid that
\[\mbf{A}\mbf{x}=\mbf{A}\left(\mbf{A}^{-}\mbf{b}\right)=
\left(\mbf{A}\mbf{A}^{-}\mbf{A}\right)\mbf{y}=\mbf{A}\mbf{y}=\mbf{b}\]
\end{proof}

\begin{remark}\label{oMP}
If the system of linear equations \eqref{slr} is solvable and its solution is expressed in the form
\(\mbf{x} = \mbf{A}^{-}\mbf{b}\) and \(\mbf{A}^{-}\) is Moore-Penrose pseudoinverse of the matrix \(\mbf{A}\), then this solution minimizes the Euclidean norm \(\left\| \mbf{Ax} - \mbf{b} \right\|\) and of all \(n\)-dimensional solutions
 \(\mbf{x}\) that minimize this norm it has the lowest norm. Such a vector solution of the system of linear equations is called the \emph{solution of the system of linear equations with the least squares solution of minimum norm} or \emph{the best approximate solution of the system}. 
\end{remark}
Note that in our considerations, when solving a solvable system of linear equations \eqref{slr}  we will consider a solution \eqref{rslr} expressed in the form
\[\mbf {x} = \mbf {A }^+ \mbf {b}\]
The solution of  \eqref{slr} exists if and only if $\mbf{A}\mbf{A}^{+}\mbf{b}=\mbf{b}$,\cite{James}.

\section{Pseudoinverse matrix differential operator} 

We shall introduce some properties of the differential and inverse differential operators.

The definition of the inverse differential operator \cite{S}: Let $\dfrac1{\phi(D)}f(x)$ be defined as a particular solution $y$, such that  \eqref{DRDp} then  $\dfrac1{\phi(D)}f(x)$ is called  \emph{the inverse differential operator}. \qed
\par We will prove relation I in the next Table 2. \\*

\proof 
 With respect to C in the Table 1 is
\begin{equation}\label{aad}
\phi(D)(xf(x)) = x\phi(D)f(x) + \phi^\prime(D)f(x)
\end{equation}
Let
\[\phi(D)f(x) = R(x)\]
Then 
\[\frac{1}{\phi(D)}\phi(D)f(x) = \frac{1}{\phi(D)}R(x)\]
\begin{equation}\label{aae}
f(x) = \frac{1}{\phi(D)}R(x)
\end{equation}
Substituting \eqref{aae} into relationship \eqref{aad}, we get 
\[\aligned
\phi(D)\left( x\frac{1}{\phi(D)}R(x) \right) &= x\phi(D)\frac{1}{\phi(D)}R(x) + \phi^\prime(D)\frac{1}{\phi(D)}R(x)\\
\phi(D)\left( x\frac{1}{\phi(D)}R(x) \right) &= x R(x) + \phi^\prime(D)\frac{1}{\phi(D)}R(x)\\
\endaligned\]
Let's multiply the previous equation from the left side with \(\dfrac{1}{\phi(D)}\) , then
\[x\frac{1}{\phi(D)}R(x) = \frac{1}{\phi(D)} x R(x) + \frac{1}{\phi(D)}\phi^\prime(D)
\frac{1}{\phi(D)}R(x)\]
From which
\[\frac{1}{\phi(D)} x R(x) = x\frac{1}{\phi(D)}R(x) - \frac{1}{\phi(D)}\phi^\prime(D)
\frac{1}{\phi(D)}R(x)\]
\qed
\pagebreak

\begin{tabular}{l}
Table 2: Inverse Operator Techniques \cite{Ch}, \cite{S}\\
\end{tabular}\\
\begin{tabular}{|l|l|}
\hline
\textbf{A.} & \\
\qquad $\frac1{\phi(D)}\sum\limits_{k=0}^nc_kR_k(x)$& 
$\sum\limits_{k=0}^nc_k\frac1{\phi(D)}R_k(x)$\\
\hline
\textbf{B.} & \\
 \(\qquad\frac1{D-m}R(x)\)	& \(e^{mx}\int e^{-mx}R(x)dx\) \\
 & \\
\hline
\textbf{C.}  & \\
&  $e^{m_1x}\int e^{-m_1x}\,e^{m_2x}\int e^{-m_2x}\dots$\\
\quad\( \frac1{(D-m_1)(D-m_2)\dots (D-m_n) }\)& $e^{m_nx}\int e^{-m_nx}R(x)dx^n$\\
& {\small This can also be evaluated by expanding} \\
& {\small the inverse operator into partial fractions}\\
& {\small and then using B.} \\
\hline
\textbf{D.}  & \\
&$\frac{e^{px}}{\phi(p)}$\, {\small if} \, $\phi(p)\ne 0$\\
 \hspace{1cm}$\frac1{\phi(D)}e^{px}$ & $\frac{x^ke^{px}}{\phi^{(k)}(p)}$\,{\small if}\,  
$\textstyle{\phi(p)=\phi^\prime(p)=\dots \phi^{(k-1)}(p)=0}$\\
&\hspace{3.9cm} {\small but} \, $\phi^{(k)}(p)\ne 0$\\
\hline
\textbf{E.}&\\
\qquad$\frac1{\phi(D^2)}\cos px$&\qquad $\dfrac{\cos px}{\phi(-p^2)}$\\
&\hspace{3cm}{\small if}\quad $\phi(-p^2)\ne 0$\\
\qquad$\frac1{\phi(D^2)}\sin px$&\qquad $\dfrac{\sin px}{\phi(-p^2)}$\\
\hline 
\textbf{F.}&\\
\qquad $\frac1{\phi(D)}\cos px$
      & \qquad $Re\left\{\frac{e^{ipx}}{\phi(ip)}\right\}$\\
$\qquad \frac1{\phi(D)}\sin px$
 & \qquad $Im\left\{\frac{e^{ipx}}{\phi(ip)}\right\}$\\
& {\small If $\phi(ip)\ne 0$, otherwise use D.}\\
\hline
\textbf{G.}&\\
\qquad $\frac1{\phi(D)}x^p$ & 
\qquad $\left(\sum\limits_{k=0}^p c_kD^k\right)x^p$\\
{\small by expanding $\frac1{\phi(D)}$ in powers of $D$}
& {\small since $D^{p+n}x^p=0$ for $n>0.$}\\
\hline
\textbf{H.}&\\
\qquad $\frac1{\phi(D)}e^{px}R(x) $&\quad $e^{px}\frac1{\phi(D+p)}R(x)$\\
& {\small  called the "operator shift theorem"}.\\
\hline
\textbf{I.}&\\
\qquad $\frac1{\phi(D)}xR(x)$ & 
$x\frac1{\phi(D)}R(x) - \frac1{\phi(D)}\phi^\prime(D)\frac1{\phi(D)}R(x)$=\\
&$\left(x-\frac1{\phi(D)}\phi^\prime(D)\right)\frac1{\phi(D)}R(x)$\\
&\\
\hline
\textbf{J.}&{\small using substitute}\\
\quad$\dfrac1{\phi(D)}\left( e^{\alpha x} P_{m}(x)\cos{\beta x} +\right.$ 
& $\cos{\beta x} = \dfrac{e^{i\beta x} + e^{- i\beta x}}{2}$\\
\qquad\quad$\left. Q_{n}(x)\sin{\beta x} \right)$&
$\sin{\beta x} = \dfrac{e^{i\beta x} - e^{- i\beta x}}{2i}$ \\
&{\small and next to use an operator shift theorem}\\
\hline
\end{tabular}
\\

Proofs of many of the statements  in Table 2 can be found for example in \cite{Ch} or \cite{S}.

Note that all analytical relations derived for the differential operator have an adequate expression for the matrix differential operator.\qed

In order to define a pseudoinverse matrix differential operator, we first consider a simple form of the differential equation \eqref{DR} i. e.
 \begin{equation}\label{sdr}
y^\prime=f(x)
\end{equation} \par

Let us assume that the function \(f(x)\) is differentiable and every derivative of the function \(f(x)\)  can be expressed as a finite linear combination of linear independent functions $f_1(x),f_2(x),\dots,f_n(x).$
Let's consider the vector space 
\[V=\text{span}\left(f_1(x),f_2(x),\dots,f_n(x)\right)\]
with the basis \(B=\{f_1(x),f_2(x),\dots,f_n(x)\}\). 

\begin{itemize}
\item[(\textbf{A})]
Let us assume that the particular solution of the differential equation \eqref{sdr} belongs to the vector space \(V\) with the basis $B$. Let  \(\mathcal{D}_B\)  be a matrix differential operator  corresponding to the basis $B$. Then the equation 
\[ \mathcal{D}_B\mbf{y}_B=\mbf{f}_B,\]
where $\mbf{f}_B=[f(x)]_B$, is solvable. And due to the Theorem \ref{SLR} 
\begin{equation}\label{byMPP}
\mbf{y}_B=\mathcal{D}_B^{+}\mbf{f}_B
\end{equation} 
is the solution of the differential equation \eqref{sdr} expressed in the basis  $B$ where 
$\mathcal{D}_B^{+}$ is the Moore-Penrose pseudoinverse of the matrix  
$\mathcal{D}_B.$\\*
Thus $\mathcal{D}_B^{+}$  is the pseudoinverse matrix differential operator to the operator  $\mathcal{D}_B$ .
(The unique solution \eqref{byMPP} to the differential equation \eqref{sdr}  do not contain a  kernel of matrix differential operator $\mathcal{D}_B^{+}$. This also applies in the following cases.)

\item[(\textbf{B})]
Let us assume that the particular solution of the differential equation \eqref{sdr} does not belong to the vector space \(V\) and let  \(\mathcal{D}_B\)  be a matrix differential operator with the considered basis B. Then the equation 
\[ \mathcal{D}_B\mbf{y}_B=\mbf{f}_B\]
is not solvable. Let's create a new system of functions 
\[\aligned
B_1&=\{f_1(x),f_2(x),\dots,f_n(x)\}\cup\{xf_1(x),xf_2(x),\dots,xf_n(x)\}\\
&=\{f_{11}(x),f_{12}(x),\dots,f_{1m}(x)\}\\
\endaligned \]
where $m>n$.
Let functions $f_{11}(x),f_{12}(x),\dots,f_{1m}(x)$ be a linear independent and let 
\[V_1=\text{span}(f_{11}(x),f_{12}(x),\dots,f_{1m}(x))\] 
with the basis $B_1$. 
Let a particular solution of the differential equation \eqref{sdr} belong to the vector space  \(V_1\) and let  \(\mathcal{D}_{B_1}\)  be a matrix differential operator with the considered basis $B_1$. Then the equation 
\begin{equation}\label{ssdr}
\mathcal{D}_{B_1}\mbf{y}_{B_1}=\mbf{f}_{B_1}
\end{equation}
is solvable, where $\mbf{f}_{B_1}=[f(x)]_{B_1}$. 
Due to Theorem \ref{SLR} 
\[\mbf{y}_{B_1}=\mathcal{D}_{B_1}^{+}\mbf{f}_{B_1}\] 
is the solution of the differential equation \eqref{sdr} expressed in the basis  $B_1$
where $\mathcal{D}_{B_1}^{+}$ is the Moore-Penrose pseudoinverse of the matrix  
$\mathcal{D}_{B_1}.$  Thus $\mathcal{D}_B^{+}$  is the pseudoinverse matrix differential operator to the operator  $\mathcal{D}_B$ . 
\item[(\textbf{C})] If a solution of the differential equation \eqref{sdr} does not belong to the vector space  \(V_1\), we will create a new system of vectors in a manner similar to $B_1$ and analyse the solvability of the differential equation. 
It is difficult to prove in general that the solution lies in some vector space, the construction of which we have described. It will be explained in the following examples.
\end{itemize}

For completeness, we still have to consider a differential equation in the form \eqref{DR}.  Also in this case, we assume that the function \(f(x)\) is differentiable and every derivative of the function \(f(x)\)  can be expressed as a finite linear combination of linear independent functions $f_1(x),f_2(x),\dots,f_n(x).$
We are considering a vector space 
\[V=\text{span}(f_1(x),f_2(x),\dots,f_n(x))\]
with the basis \(B=\{f_1(x),f_2(x),\dots,f_n(x)\}\) 
and and we discuss the existence of a solution to the equation
\begin{equation}\label{MDRD}
\left(a_{n}\mathcal{D}_B^{n} + a_{n - 1}\mathcal{D}_B^{n - 1} + \cdots + a_{1}\mathcal{D}_B + a_{0}\mbf{I}\right)\mbf{y}_B = \mbf{f}_B
\end{equation}
or briefly 
\[\phi(\mathcal{D}_B)\mbf{y}_B =\mbf{f}_B,\]
where \(\mbf{I}\) is the identity matrix. 
The discussion is analogous to previous cases. \qed\vspace{2mm}
 
Now let us present some elementary examples. We will focus this issue in more detail below. Let's start with a very elementary example.
\begin{example}
Determine using a matrix differential operator the particular solution of the equation
\begin{equation}\label{xx1}
y^{\prime}=x.
\end{equation}
\end{example} 
\noindent
\emph{Solution.} $f(x)=x$ and each derivative of $f(x)$ can be expressed as a linear combination of 
$\{x, 1\}$. Let $V=\text{span}(x,1)$. Then 
\[\mathcal{D}_B=\begin{bmatrix}
{0} & {0} \\
{1} & {0} \\
\end{bmatrix}.\]
The equation
\[\mathcal{D}_B\mbf{y}_B=[x]_B \]
\[
\begin{bmatrix}
{0} & {0} \\
{1} & {0} \\
\end{bmatrix}
\begin{bmatrix}
{y}_1 \\
{y}_2 \\
\end{bmatrix}=
\begin{bmatrix}
1 \\
0 \\
\end{bmatrix}\]
has no solution. It follows that a particular solution of \eqref{xx1} does not belong to $V$. Let's create a new system of function $B_1=\{x^2,x\}\cup\{x,1\}=\{x^2,x,1\}$ which  is linear independent. 
Then
 $$V_1=\text{span}(x^2,x,1),\quad \mathcal{D}_{B_1}=
\begin{bmatrix}
0 & 0 & 0\\
2 & 0 & 0 \\
0 & 1 & 0.\\
\end{bmatrix}
$$
The matrix equation
\[\mathcal{D}_{B_1}\mbf{y}_{B_1}=[x]_{B_1} \]
$$\left[\begin{array}{cc|c}
0 & 0 & 0\\ \hline
2 & 0 & 0 \\
0 & 1 & 0\\
\end{array}\right]\begin{bmatrix}
y_1 \\
y_2\\
y_3\\
\end{bmatrix}=\begin{bmatrix}
0 \\ 
1\\
0\\
\end{bmatrix}$$
has a solution and (Theorem \ref{SLR})
$$\mbf{y}_{B_1}=\begin{bmatrix}
y_1 \\
y_2\\
y_3\\
\end{bmatrix}=\mathcal{D}_{B_1}^{+}\begin{bmatrix}
0 \\
1\\
0\\
\end{bmatrix}$$
According to Theorem \ref{T33} 
$$\mathcal{D}_{B_1}^{+}=\left[\begin{array}{c|cc}
0 & \dfrac12 & 0\\
0 & 0 & 1 \\ \hline
0 & 0 & 0\\
\end{array}\right]$$
Thus
$$\mbf{y}_{B_1}=
\begin{bmatrix}
0 & \dfrac12 & 0\\
0 & 0 & 1 \\
0 & 0 & 0\\
\end{bmatrix}\begin{bmatrix}
0 \\
1\\
0\\
\end{bmatrix}=\begin{bmatrix}
\dfrac12 \\
0\\
0\\
\end{bmatrix}
$$
which, in turn, implies the particular solution
$$y=\frac12x^2$$

\begin{example}\label{inbase}
Determine the particular solutions of the equations
\begin{itemize}
\item[(a)] \quad $y^{\prime\prime}+3y^\prime-4y=x e^{2x}$
\item[(b)] \quad $y^{\prime\prime}+3y^\prime-4y=2x e^{2x}-3 e^{2x}$
\end{itemize}
\end{example} 
\noindent

\noindent\textit{Solution.} Every derivative of the functions on the right-hand side of the equations is a linear combination of the functions $x\cdot e^{2x}, e^{2x}$ which are linear independent  Let's assume that a particular solution of the differential equations belongs to the vector space
$V=\textrm{span}(x e^{2x}, e^{2x})$ with the base $B=\{x e^{2x}, e^{2x}\}$.  Differential operator  
\[\mathcal{D}_B=\begin{bmatrix}
{2} & {0} \\
{1} & {2} \\
\end{bmatrix}\]
\begin{itemize}
\item[(a)]
\[\left(\mathcal{D}_B^2+3\mathcal{D}_B-4\mbf{I}_2\right)
\mbf{y}_B=[x e^{2x}]_B\]
\[\left({\begin{bmatrix}
{2} & {0} \\
{1} & {2} \\
\end{bmatrix}}^2+3\begin{bmatrix}
{2} & {0} \\
{1} & {2} \\
\end{bmatrix}-4\begin{bmatrix}
{1} & {0} \\
{0} & {1} \\
\end{bmatrix}
\right)\mbf{y}_B=\begin{bmatrix}
{1} \\
 {0} \\
\end{bmatrix}\]
\[\begin{bmatrix}
{6} & {0} \\
{5} & {-4} \\
\end{bmatrix}\mbf{y}_B=\begin{bmatrix}
{1} \\
 {0} \\
\end{bmatrix}\]
\[\mbf{y}_B=\begin{bmatrix}
{6} & {0} \\
{5} & {-4} \\
\end{bmatrix}^{-1}\begin{bmatrix}
{1} \\
 {0} \\
\end{bmatrix}
\renewcommand{\arraystretch}{2}
=\begin{bmatrix}
{\dfrac16} & {0} \\
\dfrac{5}{24} & -\dfrac{1}{4} \\
\end{bmatrix}\begin{bmatrix}
{1} \\
 {0} \\
\end{bmatrix}=
\begin{bmatrix}
{\dfrac16} \\
 {\dfrac5{24}} \\
\end{bmatrix}\]
\vspace{2mm}
We have found the particular solution of differential equation in (a)
\[y=\frac16xe^{2x}+\frac5{24}e^{2x}.\]
\item[(b)]
\[\mbf{y}_B
\renewcommand{\arraystretch}{2}
=\begin{bmatrix}
{\dfrac16} & {0} \\
\dfrac{5}{24} & -\dfrac{1}{4} \\
\end{bmatrix}\begin{bmatrix}
{2} \\
 {-3} \\
\end{bmatrix}=
\begin{bmatrix}
{\dfrac13} \\
 {\dfrac76} \\
\end{bmatrix}\]
\end{itemize}
\[y=\frac13xe^{2x}+\frac76e^{2x}\]

\begin{example}\label{2times}
Find the particular solution of the differential equation
\begin{equation}\label{p2times}
y^{IV}+2y^{\prime\prime}+y=2\sin x-4\cos x
\end{equation}
\end{example}
\noindent
\textit{Solution.} It's easy to prove that the particular solution $y$ of differential equation \eqref{p2times}
\[\aligned
y\notin V&=\text{span}(\sin x, \cos x)\\
y\notin V&=\text{span}(x\sin x, x\cos x, \sin x, \cos x)\\
\endaligned \]
Let's assume that $y\in V=\text{span}(x^2\sin x, x^2\cos x, x\sin x, x\cos x, \sin x, \cos x)$. Since the set of functions
\[B=\{x^2\sin x, x^2\cos x, x\sin x, x\cos x, \sin x, \cos x\}\]
is linearly independent, it is a basis for $V$. The corresponding matrix differential operator 
\[\mathcal{D}_B=\begin{bmatrix}
0 &-1 & 0 & 0 & 0 & 0 \\
1 & 0 & 0 & 0 & 0 & 0\\
2 & 0 & 0 &-1 & 0 & 0 \\
0 & 2 & 1 & 0 & 0 & 0 \\
0 & 0 & 1 & 0 & 0 &-1\\
0 & 0 & 0 & 1 & 1& 0 \\
\end{bmatrix}\]
Then 
\[\aligned  
 \left(\mathcal{D}_B^4+2\cdot\mathcal{D}_B^2+\mbf{I}_6\right)\mbf{y}_B&=
[2\sin x-4\cos x]_B\\
\begin{bmatrix}
0 & 0 & 0 & 0 & 0 & 0\\
0 & 0 & 0 & 0 & 0 & 0\\
0 & 0 & 0 & 0 & 0 & 0\\
0 & 0 & 0 & 0 & 0 & 0\\
-8 & 0 & 0 & 0 & 0 & 0\\
0 & -8 & 0 & 0 & 0 & 0\\
\end{bmatrix}\mbf{y}_B&=\begin{bmatrix}
0\\
0\\
0\\
0\\
2\\
-4\\
\end{bmatrix}\\
\mbf{y}_B=\left[\begin{array}{cc|cccc}
0 & 0 & 0 & 0 & 0 & 0\\
0 & 0 & 0 & 0 & 0 & 0\\ 
0 & 0 & 0 & 0 & 0 & 0\\
0 & 0 & 0 & 0 & 0 & 0\\ \hline
-8 & 0 & 0 & 0 & 0 & 0\\
0 & -8 & 0 & 0 & 0 & 0\\
\end{array}\right]^{+}&\begin{bmatrix}
0\\
0\\
0\\
0\\
2\\
-4\\
\end{bmatrix}\\
\mbf{y}_B=\left[\begin{array}{cccc|cc}
0 & 0 & 0 & 0 & -\frac{1}{8} & 0\\
0 & 0 & 0 & 0 & 0 & -\frac{1}{8}\\ \hline
0 & 0 & 0 & 0 & 0 & 0\\
0 & 0 & 0 & 0 & 0 & 0\\
0 & 0 & 0 & 0 & 0 & 0\\
0 & 0 & 0 & 0 & 0 & 0\\
\end{array}\right]
\begin{bmatrix}
0\\
0\\
0\\
0\\
2\\
-4\\
\end{bmatrix}
&=\begin{bmatrix}
-\frac{1}{4}\\
\frac{1}{2}\\
0\\
0\\
0\\
0\end{bmatrix}
\endaligned \]
The Moore-Penrose pseudoinverse to the matrix $\left(\mathcal{D}_B^4+2\cdot\mathcal{D}_B^2+\mbf{I}_6\right)$ we found using  Theorem \ref{T33}. Then
\[y=-\frac14x^2\sin x+\frac12x^2\cos x\]

\section{Matrix differential operator and the method of undetermined coeficients}

In order to describe the algorithm for the finding of a particular solution of
an ordinary nonhomogeneous linear differential equation with constant
coefficients \eqref{DRD} using matrix differential operator, we use knowledge of the algorithm for finding a particular solution by an
undetermined coefficients method.

Let us consider a differential equation \eqref{DRD}  with a characteristic
equation
\begin{equation}\label{chro}
a_{n}k^{n} + a_{n - 1}k^{n - 1} + \ldots + a_{1}k + a_{0} = 0
\end{equation}
First, we make two conventions to simplify expressing the root
multiplicity of the characteristic equation and the value of the
function \(f(x) = x^{0}\) at the discontinuity point.
\begin{itemize}
\item[1.] If the number \(\alpha\) is not the root of the characteristic
equation \eqref{chro}, we will also say that it is a 0-fold root of the characteristic equation;\newline if the number \(\alpha\)  is a simple characteristic root of \eqref{chro}, we will also say that it is 1-fold root of the characteristic equation, and so on.
\item[2.] Given the removable discontinuity of the function
\(f(x) = x^{0}\) at the point $x = 0$, we define the function at this point as \(f(0) = 1.\)
\end{itemize}

Let us consider two cases of the right-hand side of the differential equation
\eqref{DRD}. We will describe the algorithm for finding a particular solution
using a matrix differential operator:
\begin{itemize}  
\item[a)] with the right-hand side \eqref{DRD}
\begin{equation}\label{epx}
f(x) = e^{\alpha x}P_{m}(x)
\end{equation}
where \(\alpha\in \mathbb{R}\) and \(P_{m}(x)\) is a polynomial of degree \(m\). The algorithm to determine the particular solution of that differential equation, by the method of undetermined coefficients, says:\newline
If $\alpha$ is the \(k\)-fold root ($k = 0,1, 2, \ldots,$) of the characteristic equation \eqref{chro} of the differential equation \eqref{DRD} with the right-hand side \eqref{epx}, then the particular solution of this equation is of the form
\begin{equation}\label{repx}
y = x^{k}e^{\alpha x}Q_{m}(x),
\end{equation}
where
\(Q_{m}(x) = A_{m}x^{m} + A_{m - 1}x^{m - 1} + \ldots + A_{1}x + A_{0}\) 
is a polynomial of degree \(m\) with undetermined coefficients.
The values of these undetermined coefficients can be determined by substituting \eqref{repx} for \(y\) into the equation \eqref{DRD} and then comparing the coefficients for the same functions on the right-hand and left-hand sides of the equation. The algorithm for determining the particular solution of the differential equation \eqref{DRD} with the right-hand side \eqref{repx} by the method of
undetermined coefficients implies that this particular solution will be
in the vector space
\[V = \text{span}( x^{k + m}e^{\alpha x},x^{k + m - 1}e^{\alpha x},\ldots,
 xe^{\alpha x},e^{\alpha x})\]
with the basis for \(V\)
\begin{equation}\label{bepq}
B = \left\{ x^{k + m}e^{\alpha x},x^{k + m - 1}e^{\alpha x},\ldots,
 xe^{\alpha x},e^{\alpha x} \right\}
\end{equation}
The relevant matrix differential operator \(\mathcal{D}_B\) in the vector
space \(V\) with the basis \(B\) be a matrix of the type
\(( k + m + 1 ) \times ( k + m + 1 ).\)
\item[b)]  with the right-hand side
\begin{equation}\label{espx}
f(x) = e^{\alpha x}\left(P_{r}(x)\sin{\beta x} + Q_{s}(x)\cos{\beta x} \right),
\end{equation}
where \(\alpha,\ \beta\) are real numbers, \(P_{r}(x)\) is a polynomial of degree \(r\), 
\(Q_{s}(x)\) is a polynomial of degree \(s\).\newline
If \(\alpha + \beta i \) is a \(k\)-fold root ($k = 0,1, 2, \ldots,$) of
the characteristic equation \eqref{chro} of the differential equation \eqref{DRD} with
the right-hand side \eqref{espx}, then the particular solution of this equation has
the form
\begin{equation}\label{respx}
y = x^{k}e^{\alpha x}\left( U_{m}(x)\sin{\beta x} + V_{m}(x)\cos{\beta x}\right) 
\end{equation}
where
\(m = \max{\{ r,s\}}\),
\(U_{m}(x) = A_{m}x^{m} + A_{m - 1}x^{m - 1} + \ldots + A_{1}x + A_{0},\), 
\(V_{m}(x) = B_{m}x^{m} + B_{m - 1\ }x^{m - 1} + \ldots + B_{1}x + B_{0}\)

are polynomials of degree \(m\) with undetermined coefficients. The values of these undetermined coefficients can be determined by substituting \eqref{respx}  for \(y\) into the equation \eqref{DRD} and then comparing the coefficients for the same functions on the right-hand and left-hand sides of the equation. The algorithm for determining the particular solution of the differential equation \eqref{DRD} with the right-hand side \eqref{respx} by the method of undetermined coefficients implies that this particular solution will be in the vector space
\[\begin{array}{l}
V = \text{span}( x^{k + m}e^{\alpha x}\sin{\beta x},x^{k + m}e^{\alpha x}\cos{\beta x},x^{k + m - 1}e^{\alpha x}\sin{\beta x},\\ x^{k + m - 1}e^{\alpha x}\cos{\beta x}\ldots,
 xe^{\alpha x}\sin{\beta x}, xe^{\alpha x}\cos{\beta x},e^{\alpha x}\sin{\beta x}, \\
 e^{\alpha x}\cos{\beta x}) 
\end{array}\]
with the basis for \(V\)
\begin{equation}\label{brespx}
\begin{array}{l}
B = \left\{ x^{k + m}e^{\alpha x}\sin{\beta x},x^{k + m}e^{\alpha x}\cos{\beta x},x^{k + m - 1}e^{\alpha x}\sin{\beta x},\right.\\ x^{k + m - 1}e^{\alpha x}\cos{\beta x}\ldots, 
 xe^{\alpha x}\sin{\beta x}, xe^{\alpha x}\cos{\beta x},e^{\alpha x}\sin{\beta x},\\
 \left.e^{\alpha x}\cos{\beta x}\right\} 
\end{array}
\end{equation}
The relevant matrix differential operator \(\mathcal{D}_B\) in the vector space \(V\) with the basis \(B\) be a matrix of the type \( 2(k + m + 1) \times 2(k + m + 1) .\)
\end{itemize}

Subsequently, in both cases a) and b), we create a matrix equation
\begin{equation}
\left( a_{n}\mathcal{D}_B^{n} + a_{n - 1}\mathcal{D}_B^{n - 1} + \cdots + a_{1}\mathcal{D}_B + a_{0}\mbf{I} \right)\mbf{y}_B = \mbf{f}_B
\end{equation}
where \(\mbf{f}_B=[f(x)]_B\) and $\mbf{y}_B=[y(x)]_B$, where $y(x)$ is a particular solution of the differential equation \eqref{DRD}. 
\vspace{2mm}

If the right-side of the differential equation \eqref{DR} is the sum of several functions, then the principle of superposition can be used to solve it. \emph{The principle of superposition of solutions} says that if $y_i$ $(i=1,2,\dots m)$ is a solution of the differential equation 
$$\left(a_{n}y^{(n)} + a_{n - 1}y^{(n - 1)} + \cdots + a_{1}y^{\prime} + a_{0}\right)y = f_i(x)$$
$(i=1,2, \dots m)$, then for any constants $k_1,k_2,\dots,k_m$, the function
$$y=k_1y_1+k_2y_2+\cdots+k_my_m$$ is a solution to the differential equation \eqref{DR} with 
$$f(x)=k_1f_1(x)+k_2f_2(x)+\cdots+k_mf_m$$

In special case, if $f_1(x),f_2(x),\dots,f_m(x)$ form the basis $B$ of the vector space $V=\text{span}(f_1(x), f_2(x),\dots,f_m(x))$ then it is enough to solve only the equation
$$\left( a_{n}\mathcal{D}_B^{n} + a_{n - 1}\mathcal{D}_B^{n - 1} + \cdots + a_{1}\mathcal{D}_B + a_{0}\mbf{I} \right)\mbf{y}_B = \begin{bmatrix}
k_1\\
k_2\\
\vdots\\
k_m
\end{bmatrix}_B$$
\vspace{2mm}

In the next example, we want to show how the matrix differential operator can be used to support the finding of a particular solution of the differential equation by the method of undetermined coefficients.

\begin{example}\label{aaf} Determine the particular solution of the differential equation
\begin{equation}\label{raaf}
(D-2)^2(D+4)^2y=3e^{2x}
\end{equation}
\end{example}

\noindent\textit{Solution.} First we solve the equation
\begin{equation}\label{aag}
(D+4)^2y=3e^{2x}
\end{equation}
Since $\alpha=2$ is not a solution of the characteristic equation $(k+4)^2=0$, it follows that the solution of \eqref{aag} belongs to the vector space $V=\text{span}(e^{2x})$ with the basis  $B_1=\{e^{2x}\}$. Then the matrix differential operator is 
\[\mathcal{D}_{B_1}=[2]\]
and
\[\aligned
(\mathcal{D}_{B_1}+4\mbf{I}_1)^2\mbf{{y}_1}_{B_1}&=[3e^{2x}]_{B_1}\\
([2]+4[1])^2\mbf{{y}_1}_{B_1}&=[3]\\
[6]^2\mbf{{y}_1}_{B_1}&=[3]\\
\mbf{{y}_1}_{B_1}&=[6]^{-2}[3]\\
\mbf{{y}_1}_{B_1}&=\left[\frac1{12}\right]\\
y_1&=\frac1{12}e^{2x}\\
\endaligned\]
Now we have to determine the particular solution of the differential equation  
\begin{equation}\label{aah}
(D-2)^2y=\frac1{12}e^{2x}
\end{equation}
Because $\alpha=2$ is 2-fold root of the characteristic equation $(k-2)^2=0$, it follows that the solution of \eqref{aah} belongs to the vector space\linebreak
$V=\text{span}(x^2e^{2x}, xe^{2x}, e^{2x})$ with the basis 
$B=\{x^2e^{2x}, xe^{2x}, e^{2x}\}$. 
Then the matrix differential operator is 
$$B=\begin{bmatrix}
2&0&0\\
2&2&0\\
0&1&2\\
\end{bmatrix} $$
and
$$\aligned
(\mathcal{D}_B-2\mbf{I}_3)^2\mbf{y}_B&=\left[\frac1{12}e^{2x}\right]_B\\
\left[\begin{array}{c|cc}
0&0&0\\
0&0&0\\\hline
2&0&0\\
\end{array}
\right]\mbf{y}_B&=\begin{bmatrix}
0\\
0\\
\dfrac1{12}\\
\end{bmatrix}\\
\mbf{y}_B&=
\left[\begin{array}{c|cc}
0&0&0\\
0&0&0\\\hline
2&0&0\\
\end{array}
\right]^+\begin{bmatrix}
0\\
0\\
\dfrac1{12}\\
\end{bmatrix}\\
\endaligned$$
\[\mbf{y}_B=
\left[\begin{array}{cc|c}
0&0&\dfrac12\\\hline
0&0&0\\
0&0&0\\
\end{array}
\right]\begin{bmatrix}
0\\
0\\
\dfrac1{12}\\
\end{bmatrix}=\begin{bmatrix}
\dfrac1{24}\\
0\\
0\\
\end{bmatrix}\]
The particular solution of the differential equation \eqref{raaf} is 
\[y=\frac1{24}x^2e^{2x}\]

\begin{example} Determine the particular solution of the differential equation
\begin{equation}\label{aab}
y^{\prime\prime} - 4y^\prime + 13y = 2xe^{2x}\cos{3x}
\end{equation}
using a pseudoinverse matrix differential operator.
\end{example}
\noindent\textit{Solution.} Because  $2 +3i$ is 1-fold root of the characteristic equation\linebreak
$k^2-4k+13=0$, it follows that the particular solution of \eqref{aab} will be (due to \eqref{espx},\eqref{respx}) in the form 
\[y = xe^{2x}\left( (Ax + B)\sin{\beta x} + (Cx + D)\cos{\beta x} \right)\]
In other words, the particular solution \(y\) will be in the vector space
\[V = \text{span}(x^{2}e^{2x}\sin{3x},x^{2}e^{2x}\cos{3x},xe^{2x}\sin{3x},xe^{2x}\cos{3x},e^{2x}\sin{3x}, e^{2x}\cos{3x})\]
with the basis
\[B = \left\{ x^{2}e^{2x}\sin{3x},x^{2}e^{2x}\cos{3x},xe^{2x}\sin{3x},xe^{2x}\cos{3x},e^{2x}\sin{3x}, e^{2x}\cos{3x} \right\}\]
The matrix differential operator
\[\mathcal{D}_B=\begin{bmatrix}
2 & - 3 & 0 & 0 & 0 & 0 \\
3 & 2 & 0 & 0 & 0 & 0 \\
2 & 0 & 2 & - 3 & 0 & 0 \\
0 & 2 & 3 & 2 & 0 & 0 \\
0 & 0 & 1 & 0 & 2 & - 3 \\
0 & 0 & 0 & 1 & 3 & 2 \\
\end{bmatrix}\]

\[\left( \mathcal{D_B}^{2} - 4\mathcal{D}_B +13\mbf{I}_{6} \right)\mbf{y}_B =[2xe^{2x}\cos{3x}]_B\]
After editing the previous equation, we get
\[
\left[\begin{array}{cccc|cc}
0 & 0 & 0 & 0 & 0 & 0\\
0 & 0 & 0 & 0 & 0 & 0\\\hline
0 & -12 & 0 & 0 & 0 & 0\\
12 & 0 & 0 & 0 & 0 & 0\\
2 & 0 & 0 & -6 & 0 & 0\\
0 & 2 & 6 & 0 & 0 & 0\\
\end{array}\right]
\mbf{y}_B=\begin{bmatrix}
0 \\
0 \\
0 \\
2 \\
0 \\
0 \\
\end{bmatrix}\]
\[\mbf{y}_B=\left[\begin{array}{cccc|cc}
0 & 0 & 0 & 0 & 0 & 0\\
0 & 0 & 0 & 0 & 0 & 0\\\hline
0 & -12 & 0 & 0 & 0 & 0\\
12 & 0 & 0 & 0 & 0 & 0\\
2 & 0 & 0 & -6 & 0 & 0\\
0 & 2 & 6 & 0 & 0 & 0\\
\end{array}\right]^{+}
\begin{bmatrix}
0 \\
0 \\
0 \\
2 \\
0 \\
0 \\
\end{bmatrix}\]

In accordance with Theorem \ref{T33} we determine the Moore-Penrose pseudoinverse. Then
\[\mbf{y}_B=
\renewcommand{\arraystretch}{1.9}
\left[\begin{array}{cc|cccc}
0 & 0 & 0 & \dfrac{1}{12} & 0 & 0\\
0 & 0 & -\dfrac{1}{12} & 0 & 0 & 0\\
0 & 0 & \dfrac{1}{36} & 0 & 0 & \dfrac{1}{6}\\
0 & 0 & 0 & \dfrac{1}{36} & -\dfrac{1}{6} & 0\\\hline
0 & 0 & 0 & 0 & 0 & 0\\
0 & 0 & 0 & 0 & 0 & 0\\
\end{array}\right]
\begin{bmatrix}
0 \\
0 \\
0 \\
2 \\
0 \\
0 \\
\end{bmatrix}=\begin{bmatrix}
\dfrac{1}{6}\\
0\\
0\\
\dfrac{1}{18}\\
0\\
0\\
\end{bmatrix}\]
So, the particular solution is
\[y_{p} = \frac{1}{6}x^{2}e^{2x}\sin{3x} + \frac{1}{18} xe^{2x}\cos{3x}\qed \]

\begin{example} Find the integral
\[\int\left( 13xe^{2x}\sin{3x} - 13xe^{2x}\cos{3x} + 5e^{2x}\sin{3x} - 4e^{2x}\cos{3x} \right)dx\]
\end{example}
\noindent\textit{Solution}. We want to solve the differential equation
\[Dy = 13xe^{2x}\sin{3x} - 13xe^{2x}\cos{3x} + 5e^{2x}\sin{3x} - 4e^{2x}\cos{3x}\]
The solution will be in the vector space
\[V = \text{span}(xe^{2x}\sin{3x,}xe^{2x}\cos{3x,}e^{2x}\sin{3x,}e^{2x}\cos{3x})\]
with the basis
\[B =  \left\{ xe^{2x}\sin{3x},xe^{2x}\cos{3x},e^{2x}\sin{3x},e^{2x}\cos{3x} \right\}\]
In \eqref{schemat} we calculated
\[\mathcal{D}_B =\begin{bmatrix}
2 & - 3 & 0 & 0 \\
3 & 2 & 0 & 0 \\
1 & 0\  & 2 & - 3 \\
0 & 1 & 3 & 2 \\
\end{bmatrix}\]
We are having to solve the matrix equation
\[\mathcal{D}_B\mbf{y}_B=\begin{bmatrix}
13 \\
 - 13 \\
5 \\
 - 4 \\
\end{bmatrix}\]
Then
\[\mbf{y}_B=\mathcal{D}_B^{-1}\begin{bmatrix}
13 \\
 - 13 \\
5 \\
 - 4 \\
\end{bmatrix} = 
\renewcommand{\arraystretch}{1.9}
\left[\begin{array}{cccc}
\dfrac{2}{13} & \dfrac{3}{13} & 0 & 0 \\
 - \dfrac{3}{13} & \dfrac{2}{13} & 0 & 0 \\
\dfrac{5}{169} & - \dfrac{12}{169} & \dfrac{2}{13} & \dfrac{3}{13} \\
\dfrac{12}{169} & \dfrac{5}{169} & - \dfrac{3}{13} & \dfrac{2}{13} \\
\end{array}\right]\begin{bmatrix}
13 \\
 - 13 \\
5 \\
 - 4 \\
\end{bmatrix} = \begin{bmatrix}
 - 1 \\
 - 5 \\
\dfrac{15}{13} \\
- \dfrac{16}{13} \\
\end{bmatrix}\]

We have calculated that
\[\aligned
\int{\left( 13xe^{2x}\sin{3x} - 13xe^{2x}\cos{3x} + 5e^{2x}\sin{3x} - 4e^{2x}\cos{3x} \right)dx} \\ 
=- xe^{2x}\sin{3x} - 5xe^{2x}\cos{3x} + \frac{15}{13}e^{2x}\sin{3x} - \frac{16}{13}e^{2x}\cos{3x} + C
\endaligned\]\qed

Sometimes it is useful to combine a differential operator with a matrix differential operator. 
Relationship I in Table 2 allows us to reduce a matrix differential operator by two rows and two columns. In the following we will give an example of this.

\begin{example}{\label{domo}} Find a particular solution of 
\begin{equation}\label{edomo}
(D^{2} - 5D + 16)y=xe^{2x}\sin{3x}
\end{equation}
\end{example}
\noindent\emph{Solution.} Using I in the Table 2 we have
\begin{equation}\label{rdomo}
\begin{aligned}
y=&\frac1{D^{2} - 5D + 16}xe^{2x}\sin{3x} \\
&=\left( x - \frac1{D^{2} - 5D + 16}(2D - 5) \right)\frac1{D^{2} - 5D + 16}e^{2x}\sin{3x}\\
\end{aligned}
\end{equation}

All we need to do is to create a $2\times 2$ matrix differential operator instead of $4\times 4$.  The particular solution of \eqref{edomo} belongs to the vector space \linebreak
\(V = \text{span}\left( xe^{2x}\sin{3x}, xe^{2x}\cos{3x},e^{2x}\sin{3x}, e^{2x}\cos{3x}\right)\). It can be reduced using \eqref{rdomo} to the vector space
\(V = \text{span}\left(e^{2x}\sin{3x}, e^{2x}\cos{3x}\right)\) with the basis 
\(B = \{e^{2x}\sin{3x}, e^{2x}\cos{3x}\}\). \par
Then
\[\mathcal{D}_B =\begin{bmatrix}
2 & - 3 \\
3 & 2 \\
\end{bmatrix}\]
\[\mathcal{D}_B^{2}- 5\mathcal{D}_B+16\mbf{I}_{2}=\begin{bmatrix}
1 & 3 \\
- 3 & 1 \\
\end{bmatrix}\]
\[\left(\mathcal{D}_B^{2}- 5\mathcal{D}_B+16\mbf{I}_{2}\right)^{-1}=
\renewcommand{\arraystretch}{2}
\begin{bmatrix}
\dfrac{1}{10} & - \dfrac{3}{10} \\
\dfrac{3}{10} & \dfrac{1}{10} \\
\end{bmatrix}\]
\[\left(2\mathcal{D}_B - 5\mbf{I}_{2}\right)^{-1}=\begin{bmatrix}
- 1 & - 6 \\
6 & - 1 \\
\end{bmatrix}\]
where \(\mbf{I}_{2}\) is the $2\times 2$ identity matrix.
The right-hand side of the relationship \eqref{rdomo} expressed by the differential operator is

\[\left( x\mbf{I}_{2} - \left( \mathcal{D}_B^{2}- 5\mathcal{D}_B+16\mbf{I}_{2} 
\right)^{- 1}(2\mathcal{D}_B - 5\mbf{I}_{2} \right)\left( \mathcal{D}_B^{2}- 5
\mathcal{D}_B+16\mbf{I}_{2} \right)^{- 1}\begin{bmatrix}
1 \\
0 \\
\end{bmatrix}=\]
\[\left( x \renewcommand{\arraystretch}{2}
\begin{bmatrix}
1 & 0 \\
0 & 1 \\
\end{bmatrix} - \begin{bmatrix}
\dfrac{1}{10} & - \dfrac{3}{10} \\
\dfrac{3}{10} & \dfrac{1}{10} \\
\end{bmatrix}\begin{bmatrix}
 - 1 & - 6 \\
6 & - 1 \\
\end{bmatrix} \right) \renewcommand{\arraystretch}{2}
\begin{bmatrix}
\dfrac{1}{10} & - \dfrac{3}{10} \\
\dfrac{3}{10} & \dfrac{1}{10} \\
\end{bmatrix}\begin{bmatrix}
1 \\
0 \\
\end{bmatrix}=\]
\[\left( \begin{bmatrix}
x & 0 \\
0 & x \\
\end{bmatrix} - \renewcommand{\arraystretch}{2}
\begin{bmatrix} 
\dfrac{1}{10} & - \dfrac{3}{10} \\
\dfrac{3}{10} & \dfrac{1}{10} \\
\end{bmatrix}\begin{bmatrix}
 - 1 & - 6 \\
6 & - 1 \\
\end{bmatrix} \right)\renewcommand{\arraystretch}{2}
\begin{bmatrix}
\dfrac{1}{10} \\
\dfrac{3}{10} \\
\end{bmatrix} =\]
\[\renewcommand{\arraystretch}{2}
\begin{bmatrix}
\dfrac{1}{10}x \\
\frac{3}{10}x \\
\end{bmatrix} -
 \begin{bmatrix}\renewcommand{\arraystretch}{2}
\dfrac{1}{10} & - \dfrac{3}{10} \\
\dfrac{3}{10} & \dfrac{1}{10} \\
\end{bmatrix}\renewcommand{\arraystretch}{2}
\begin{bmatrix}
 - \dfrac{19}{10} \\
\dfrac{3}{10} \\
\end{bmatrix} = \renewcommand{\arraystretch}{2}
\begin{bmatrix}
\dfrac{1}{10}x + \dfrac{7}{25} \\
\dfrac{3}{10}x + \dfrac{27}{50} \\
\end{bmatrix}\]

We can rewrite this using functions as
\[\frac{1}{D^2 - 5D + 16}xe^{2x}\sin{3x} = \left( \frac{1}{10}x + \frac{7}{25} \right)e^{2x}\sin{3x} + \left( \frac{3}{10}x + \frac{27}{50} \right)e^{2x}\cos{3x}\]
so the particular solution of the differential equation \eqref{edomo} is
\[y=\left( \frac{1}{10}x + \frac{7}{25} \right)e^{2x}\sin{3x} + \left( \frac{3}{10}x + \frac{27}{50} \right)e^{2x}\cos{3x}\] \qed

In the previous example, the number resulting from the right-hand side of the differential equation  \eqref{edomo} was not the root of the corresponding characteristic equation of \eqref{edomo}.
In the next example it will be.  

\begin{example}\label{coi} Determine the~particular solution of the equation
\begin{equation}\label{ecoi}
\left( D^{2} + 1 \right)y = x\cos x
\end{equation}
\end{example}
\emph{Solution.} The complex number $ i$ is a single root of the characteristic equation \( k^{2} + 1 = 0\), therefore the particular solution of the differential equation
 \eqref{ecoi} will be in the vector space
\[V =\textrm{span}\left(x^{2}\sin{x}, x^{2}\cos{x},  x\sin x,  x\cos{x},\sin x,\cos{x}\right)\]
with the basis for $V$ 
 \[B = \{x^2\sin{x}, x^{2}\cos x, x\sin x, x\cos{x},\sin x,\cos x\}\]
The pseudoinverse matrix differential operator will be a $6\times 6$ matrix.  
We will present three methods of solution: using a matrix differential operator, a combined method  and using a complex variable\\
a) using a matrix differential operator\\
We have the matrix differential operator
\[\mathcal{D}_B=\begin{bmatrix}
0 & -1 & 0 & 0 & 0 & 0\\
1 & 0 & 0 & 0 & 0 & 0\\
2 & 0 & 0 & -1 & 0 & 0\\
0 & 2 & 1 & 0 & 0 & 0\\
0 & 0 & 1 & 0 & 0 & -1\\
0 & 0 & 0 & 1 & 1 & 0\end{bmatrix}\]
\[(\mathcal{D}_B^2+\mbf{I}_6)\mbf{y}_B=[x\cos x]_B\]
\[\left[\begin{array}{cccc|cc}
0 & 0 & 0 & 0 & 0 & 0\\
0 & 0 & 0 & 0 & 0 & 0\\\hline
0 & -4 & 0 & 0 & 0 & 0\\
4 & 0 & 0 & 0 & 0 & 0\\
2 & 0 & 0 & -2 & 0 & 0\\
0 & 2 & 2 & 0 & 0 & 0
\end{array}\right]\mbf{y}_B=
\begin{bmatrix}0\\
0\\
0\\
1\\
0\\
0\end{bmatrix}\]
Using a Moore-Penrose pseudoinverse matrix (Theorem \ref{T33} and \ref{SLR}) we have
\[\mbf{y}_B=(\mathcal{D}_B^2+\mbf{I}_6)^+[x\cos x]_B\]
\[\mbf{y}_B=
\left[\begin{array}{cc|cccc}
0 & 0 & 0 & \frac{1}{4} & 0 & 0\\
0 & 0 & -\frac{1}{4} & 0 & 0 & 0\\
0 & 0 & \frac{1}{4} & 0 & 0 & \frac{1}{2}\\
0 & 0 & 0 & \frac{1}{4} & -\frac{1}{2} & 0\\\hline
0 & 0 & 0 & 0 & 0 & 0\\
0 & 0 & 0 & 0 & 0 & 0
\end{array}\right]
\begin{bmatrix}0\\
0\\
0\\
1\\
0\\
0\end{bmatrix}=
\begin{bmatrix}\frac{1}{4}\\
0\\
0\\
\frac{1}{4}\\
0\\
0\end{bmatrix}\]
So the particular solution of \eqref{ecoi} is 
\[y=\frac14x^2\sin x+\frac14x\cos x\]
b) We can reduce the size of the matrix $\mathcal{D}_B$ by using I in the Table 2.
\[\frac{1}{D^{2} + 1}x\cos x = x\frac{1}{D^{2} + 1}\cos x - \frac{1}{D^{2} + 1}2D
\frac{1}{D^{2} + 1}\cos x \]

Let's solve \(\dfrac{1}{D^{2} + 1}\cos x\) which corresponds to the particular solution of the differential equation $(D^2+1)y=\cos x$. This particular solution belongs to the vector space
\[V = \text{span}(x\sin x, x\cos x, \sin x, \cos x) \]
with the basis 
\(B = \{ x\sin x, x\cos x, \sin x, \cos x\} \)
Then
\[\mathcal{D}_{B}=\begin{bmatrix}
0 & - 1 & 0 & 0 \\
1 & 0 & 0 & 0 \\
1 & 0 & 0 & - 1 \\
0 & 1 & 1 & 0 \\
\end{bmatrix}\]
\[\aligned
 \mathcal{D}^2_{B}+\mbf{I}_4)\mbf{y}_B&=[\cos x]_B\\
\left[\begin{array}{cc|cc}
0 & 0 & 0 & 0 \\
0 & 0 & 0 & 0 \\\hline
0 & - 2 & 0 & 0 \\
2 & 0 & 0 & 0 \\
\end{array}\right]\mbf{y}_B&=\begin{bmatrix}
0 \\
0 \\
0 \\
1 \\
\end{bmatrix}
\endaligned\]
\[\aligned
\mbf{y}_B&=\left[\begin{array}{cc|cc}
0 & 0 & 0 & 0 \\
0 & 0 & 0 & 0 \\\hline
0 & - 2 & 0 & 0 \\
2 & 0 & 0 & 0 \\
\end{array}\right]^+\begin{bmatrix}
0 \\
0 \\
0 \\
1 \\
\end{bmatrix}\\
\mbf{y}_B&=\left[\begin{array}{cc|cc}
0 & 0 & 0 & \frac{1}{2} \\
0 & 0 & - \frac{1}{2} & 0 \\\hline
0 & 0 & 0 & 0 \\
0 & 0 & 0 & 0 \\
\end{array}\right]\begin{bmatrix}
0 \\
0 \\
0 \\
1 \\
\end{bmatrix}= 
\begin{bmatrix}
\frac{1}{2} \\
0 \\
0 \\
0 \\
\end{bmatrix}
\endaligned \]

We have calculated that
\[\frac{1}{D^{2} + 1}\cos x = \frac{x\sin x}{2}\]
Similarly
\[\left(\frac{1}{D^{2} + 1}\sin x\right)_B =
\left[\begin{array}{cc|cc}
0 & 0 & 0 & \frac{1}{2} \\
0 & 0 & - \frac{1}{2} & 0 \\\hline
0 & 0 & 0 & 0 \\
0 & 0 & 0 & 0 \\
\end{array}\right]\begin{bmatrix}
0 \\
0 \\
1 \\
0 \\
\end{bmatrix}= 
\begin{bmatrix}
0 \\
-\frac12 \\
0 \\
0 \\
\end{bmatrix}\]
So 
\[\frac{1}{D^{2} + 1}\sin x=-\frac{x\cos x}x\] 
Let's continue to calculate
\[\frac{1}{D^{2} + 1}x\cos x = x\frac{x\sin x}{2}- \frac{1}{D^{2} + 1}2D
\frac{1}{D^{2} + 1}\cos x \]
\[\frac{1}{D^{2} + 1}x\cos x = x\frac{x\sin x}{2}- \frac{1}{D^{2} + 1}(\sin x +x\cos x) \]
\[\frac{1}{D^{2} + 1}x\cos x = x\frac{x\sin x}{2}+\frac{x\cos x}2 -\frac{1}{D^{2} + 1}x\cos x\]
From this, by expressing $\dfrac{1}{D^{2} + 1}x\cos x$, we have
\[\frac{1}{D^{2} + 1}x\cos x = \frac{x^2\sin x}{4}+\frac{x\cos x}4 \]

\noindent c) using a complex variable. To use J in the Table 2.  See also \cite{Ch}.
\[\aligned
y &= \frac{1}{D^{2} + 1}x\cos x =\frac{1}{D^{2} + 1}\left( x\frac{e^{ix} + e^{- ix}}2 \right) \\
&=\frac{1}{2}\left( e^{ix}\frac{1}{( D + i )^{2} + 1}x + e^{- ix}\frac{1}{( D - i )^{2} + 1}x \right) \\
&=\frac{1}{2}\left( e^{ix}\frac{1}{D( D + 2i)}x + e^{- ix}\frac{1}{D( D - 2i)}x \right)\\
&=\frac{1}{2}\left( e^{ix}\frac{1}{D}\left( - \frac{i}{2} + \frac{D}{4} \right)x + e^{- ix}
\frac{1}{D}\left( \frac{i}{2} + \frac{D}{4} \right)x \right)\\
&= \frac{1}{2}\left( e^{ix}\frac{1}{D}\left( - \frac{ix}{2} + \frac{1}{4} \right) + e^{- ix}
\frac{1}{D}\left( \frac{ix}{2} + \frac{1}{4}\right) \right)\\
&=\frac{1}{2}\left( e^{ix}\left( - \frac{ix^{2}}{4} + \frac{x}{4} \right) + e^{- ix}
\left( \frac{ix^{2}}{4} + \frac{x}{4} \right) \right) \\
&=\frac{ix^2}{4}\frac{- e^{ix} + e^{- ix}}{2} + \frac{x}{4}\frac{e^{ix} + e^{- ix}}{2}\\
&=\frac{x^2}{4}\frac{e^{ix} - e^{- ix}}{2i} + \frac{x}{4}\frac{e^{ix} + e^{- ix}}{2} = \frac{x^2\sin x}{4} + \frac{x\cos x}{4}
\endaligned \] \qed \par

In the previous example we compared the solution method with other methods. In the following example, we return to the differential operator for solving a linear ODE with a right-hand polynomial function. In this case, it is proposed to expand the operator $\frac1{\phi (D)}$ in powers of $D$. For example, in \cite{Ch} for  
\[\frac1{\phi(D)}f(x)=\frac1{a_nD^n+a_{n-1}D^{n-1}+\ldots+a_1D+a_0}f(x)\]
is the expand for $a_0\ne 0$
\begin{equation}\label{aai}
y  = \frac{1}{\phi(D)}f(x) = 
\sum_{k = 0}^{\infty}(-1)^{k}\frac{1}{a_0}\left( \frac{a_n}{a_0}D^n + \ldots + 
\frac{a_1}{a_0}D \right)^{k}f(x)
\end{equation}
In the case
\(a_{0} = a_{1} = \ldots = a_{k - 1} = 0\quad  (1 \leq n \leq n))\ \text{and }\ (a_{k} \neq 0\), then
\begin{equation}\label{aaj}
y =\frac{1}{\phi(D)}f(x) = D^{- k}\frac{1}{a_{n}D^{n - k} + a_{n - 1}D^{n - k - 1} + \ldots + a_{k}}f(x),
\end{equation}
where we would expand an inverse operator as a Maclaurin series in the sense
of \eqref{aai}.\\
The disadvantage \eqref{aai} of expanding as a Maclaurin series is the calculation
$\left( \dfrac{a_n}{a_0}D^n + \ldots + \dfrac{a_1}{a_0}D \right)^{k}$  although it is not necessary to count all members of the power. We will show a different approach to solving this problem.

\begin{theorem}\label{mcr}
 Let \(a_{0} \neq 0\), then the Maclaurin expansion
\[\frac{1}{a_{n}D^{n} + a_{n - 1}D^{n - 1} + \cdots + a_{1}D + a_{0}} = c_{0} + c_{1}D + c_{2}D^{2} + \ldots\]
where \(c_{k} = 0\)  for \(k < 0\), \(c_{0} = \dfrac{1}{a_{0}}\), 
\(c_{k} = \mbf{q}\cdot ( c_{k - n},c_{k - n + 1},\ldots,c_{k - 1})\), for \(k = 1,2,\ldots \)
is a dot product of  vectors\,
\(\mbf{q} =\dfrac{1}{a_{0}}(-a_{n}, - a_{n - 1},\ldots,-a_1)\), \\* \( (c_{k - n},c_{k - n + 1},\ldots,c_{k - 1}).\)
\end{theorem}
\begin{proof} 
The statement follows from the identity
\begin{equation}\label{rmcr}
1 = \left( a_{n}D^{n} + a_{n - 1}D^{n - 1} + \cdots + a_{1}D + a_{0} \right)\left( c_{0} + c_{1}D + c_{2}D^{2} + \ldots \right)
\end{equation}
after expanding the right-hand side of \eqref{rmcr} and comparing coefficients of the same powers of \(D\) we get
\[c_{0} = \frac{1}{a_{0}},\]
The coefficient of \(D^{k},\ k = 1,2,3,\ldots\) we get from the equation
\[c_{k}a_{0} + c_{k - 1}a_{1} + c_{k - 2}a_{2} + \ldots + c_{k - n + 1}a_{n - 1} + c_{k - n}a_{n} = 0,\]
where \(c_{k} = 0\) for \(k < 0.\)
From this, we immediately have the proof of  the Theorem \ref{mcr}. 
\end{proof}

\begin{example}
Find a particular solution of the equation
\begin{equation}\label{aak}
(D^3-D^2+2D+1)y=x^3+2x^2+3x
\end{equation} 
using the Maclaurin expansion of the inverse differential operator.
\end{example}
\noindent\emph{Solution.}
Compute the coefficients of the Maclaurin expansion 
\[ \aligned
c_{0} &= 1\\
c_{1} &= ( -1,1, -2)\cdot( c_{- 2},c_{- 1},c_{0}) =( -1,1, -2)\cdot( 0,0,1) = - 2\\
c_{2} &= ( -1,1, -2)\cdot( c_{- 1},c_{0},c_{1}) = ( -1,1, -2)\cdot(0,1,-2)  = 5\\
c_{3} &= ( -1,1, -2 )\cdot( c_{0},c_{1},c_{2} ) = ( -1,1, -2 )\cdot( 1, -2,5) = -13\\
c_{4} &=( -1,1, -2)\cdot( c_{1},c_{2,}c_{3}) =(-1,1, -2)\cdot( -2, 5, - 13 ) = 33\\
\vdots&\\
\endaligned\]
Then we have
\[\aligned
&\frac1{D^3-D^2+2D+1}(x^3+2x^2+3x)\\
&=(1-2D+5D^2-13D^3+33D^4+\cdots)(x^3+2x^2+3x)\\
&=(1-2D+5D^2-13D^3)x^3+(1-2D+5D^2)2x^2+(1-2D)3x\\
&=(x^3-6x^2+30x-78)+(2x^2-8x+20)+(3x-6)\\
&=x^3-4x^2+25x-64\\
\endaligned\]
The particular solution of the differential equation \eqref{aak}
\[y=x^3-4x^2+25x-64\]\par
We will also present a shortened calculation using a matrix differential ope\-rator.
The solution of \eqref{aak} belongs to the vector space \(V=\textrm{span}(x^3,x^2,x,1)\) with
basis \(B=\{x^3,x^2,x,1\}\). Then 
\[\mathcal{D }_B=\begin{bmatrix}
0&0&0&0\\
3&0&0&0\\
0&2&0&0\\
0&0&1&0\\
\end{bmatrix},\quad
\mathcal{D }_B^2=\begin{bmatrix}
0&0&0&0\\
0&0&0&0\\
6&0&0&0\\
0&2&0&0\\
\end{bmatrix},\quad
\mathcal{D }_B^3=\begin{bmatrix}
0&0&0&0\\
0&0&0&0\\
0&0&0&0\\
6&0&0&0\\
\end{bmatrix}\]
\[\mathcal{D }_B^n=\begin{bmatrix}
0&0&0&0\\
0&0&0&0\\
0&0&0&0\\
0&0&0&0\\
\end{bmatrix},\quad \textrm{for}\quad n>3 \]
\[\aligned
\mbf{y}_B&=(\mbf{I}_4-2\mathcal{D }_B+5\mathcal{D }_B^2-13\mathcal{D }_B^3)
[x^3+2x^2+3x]_{B}\\
&=\begin{bmatrix}1 & 0 & 0 & 0\\
-6 & 1 & 0 & 0\\
30 & -4 & 1 & 0\\
-78 & 10 & -2 & 1\end{bmatrix}
\begin{bmatrix}
1\\
2\\
3\\
0\\
\end{bmatrix}=\begin{bmatrix}
1\\
-4\\
25\\
-64
\end{bmatrix}\\
\endaligned\]
The particular solution
\[y=x^3-4x^2+25x-64\] \qed

\begin{example}\label{blc}Using block matrices determine a particular solution of the differential equation
\begin{equation}\label{rblc}
(D^2 + 6D + 13){y} = 2xe^{- 3x}\sin{2x} - 4xe^{- 3x}\cos{2x}
\end{equation}
\end{example}
\noindent\emph{Solution}. Since the roots of the characteristic equation
are \(- 3 +  2i,  - 3 - 2i\) compared to the right-hand side of the
differential equation, it follows that the particular solution will be in a vector space
\[\aligned
V =& \textrm{span}\left(x^{2}e^{- 3x}\sin{2x},x^{2}e^{- 3x}\cos{3x},xe^{- 3x}\sin{2x},
xe^{- 3x}\cos{3x},\right. \\
&\left.e^{- 3x}\sin{2x}, e^{- 3x}\cos{3x}\right)\\ 
\endaligned \]

with basis B of V
\[\aligned
B =& \left\{x^{2}e^{- 3x}\sin{2x},x^{2}e^{- 3x}\cos{3x},xe^{- 3x}\sin{2x},
xe^{- 3x}\cos{3x},\right. \\
&\left.e^{- 3x}\sin{2x}, e^{- 3x}\cos{3x}\right\}\\ 
\endaligned \]
\[\mathcal{D }_B=
\left[\begin{array}{cc|cc|cc}
 - 3 & - 2 & 0 & 0 & 0 & 0 \\
2 & - 3 & 0 & 0 & 0 & 0 \\\hline
2 & 0 & - 3 & - 2 & 0 & 0 \\
0 & 2 & 2 & - 3 & 0 & 0 \\\hline
0 & 0 & 1 & 0 & - 3 & - 2 \\
0 & 0 & 0 & 1 & 2 & - 3 \\
\end{array}\right]\]
We express this matrix as a block matrix
\[\mathcal{D}_B=\begin{bmatrix}
\mbf{C} & \mbf{0} & \mbf{0} \\
\mbf{2}\mbf{I} & \mbf{C} & \mbf{0} \\
\mbf{0} & \mbf{I} & \mbf{C} \\
\end{bmatrix}\]
where
\[\mbf{C }=\begin{bmatrix}  - 3 & - 2 \\ 2 & - 3 \\ \end{bmatrix},\quad
\mbf{I} = \begin{bmatrix} 1 & 0 \\ 0 & 1 \\ \end{bmatrix},\quad
\mbf{0} = \begin{bmatrix} 0 & 0 \\ 0 & 0 \\ \end{bmatrix}\]
Using mathematical induction it can be proved that
\begin{equation}\label{mobl}
\mathcal{D}_B^{n}=\begin{bmatrix}
\mbf{C}^{n} & \mbf{0} & \mbf{0} \\
2n\mbf{C}^{n - 1} & \mbf{C}^{n} & \mbf{0} \\
{n(n - 1)\mbf{C}}^{n - 2} & n\mbf{C}^{n - 1} & \mbf{C}^{n} \\
\end{bmatrix},\quad n = 0,1,2,\ldots
\end{equation}
It is easy verify that equation \eqref{mobl} is true also for $n=-1,-2,\ldots$, i.e.
\[\mathcal{D}_B^{- n}=\begin{bmatrix}
\mbf{C}^{- n} & \mbf{0} & \mbf{0} \\
 - 2n\mbf{C}^{- n - 1} & \mbf{C}^{- n} & \mbf{0} \\
n(n + 1)\mbf{C}^{- n - 2} & - n\mbf{C}^{- n - 1} & \mbf{C}^{- n} \\
\end{bmatrix}\quad n = 1,2,3,\ldots\]

Then
{\small \[\aligned 
&\mathcal{D}_B^{2}+6\mathcal{D}_B+13\begin{bmatrix}
\mbf{I} & \mbf{0} & \mbf{0} \\
\mbf{0} & \mbf{I} & \mbf{0} \\
\mbf{0} & \mbf{0} & \mbf{I} \\
\end{bmatrix}=\begin{bmatrix}
\mbf{C}^{2} & \mbf{0} & \mbf{0} \\
4\mbf{C} & \mbf{C}^{2} & \mbf{0} \\
2\mbf{I} & 2\mbf{C} & \mbf{C}^{2} \\
\end{bmatrix}+6\begin{bmatrix}
\mbf{C} & \mbf{0} & \mbf{0} \\
\mbf{2}\mbf{I} & \mbf{C} & \mbf{0} \\
\mbf{0} & \mbf{I} & \mbf{C} \\
\end{bmatrix}+13\begin{bmatrix}
\mbf{I} & \mbf{0} & \mbf{0} \\
\mbf{0} & \mbf{I} & \mbf{0} \\
\mbf{0} & \mbf{0} & \mbf{I} \\
\end{bmatrix}\\
& \mbf=\begin{bmatrix}
\mbf{C}^{2}+6\mbf{C} +13\mbf{I} & \mbf{0} & \mbf{0} \\
4\mbf{C} +12\mbf{I} & \mbf{C}^{2}+6\mbf{C} +13\mbf{I} & \mbf{0} \\
2\mbf{I} & 2\mbf{C }+6\mbf{I} & \mbf{C}^{2}+6\mbf{C }+13\mbf{I} \\
\end{bmatrix}=\begin{bmatrix}
\mbf{0} & \mbf{0} & \mbf{0} \\
4\mbf{C} +12\mbf{I} & \mbf{0} & \mbf{0} \\
2\mbf{I} & 2\mbf{C} +6\mbf{I} & \mbf{0} \\
\end{bmatrix}\\
\endaligned \]}

because
\[\mbf{C}^{2}+6\mbf{C} +13\mbf{I }=\begin{bmatrix}
5 & 12 \\
 - 12 & 5 \\
\end{bmatrix} + 6\begin{bmatrix}
 - 3 & - 2 \\
2 & - 3 \\
\end{bmatrix} + 13\begin{bmatrix}
1 & 0 \\
0 & 1 \\
\end{bmatrix} = \begin{bmatrix}
0 & 0 \\
0 & 0 \\
\end{bmatrix} = \mbf{0}\]

Next
\[2\mbf{C} +6\mbf{I} =\begin{bmatrix}
0 & - 4 \\
4 & 0 \\
\end{bmatrix}\]
\[4\mbf{C} +12\mbf{I }=2\left( 2\mbf{C} +6\mbf{I} \right) = 
\begin{bmatrix}
0 & - 8 \\
8 & 0 \\
\end{bmatrix}\]
\[2\mbf{I} =\begin{bmatrix}
2 & 0 \\
0 & 2 \\
\end{bmatrix}\]

Dividing the matrix into $2\times 2$ blocks, we have
\[\left[\begin{array}{cc|c}
\mbf{0} & \mbf{0} & \mbf{0} \\\hline
4\mbf{C} +12\mbf{I} & \mbf{0} & \mbf{0} \\
2\mbf{I} & 2\mbf{C} +6\mbf{I} & \mbf{0} \\
\end{array}\right]\]

From Theorem \ref{T33}, we have

\begin{equation}\label{aac}
\aligned
\mbf{y}_B &= \begin{bmatrix}
\mbf{0} & \mbf{0} & \mbf{0} \\
4\mbf{C} +12\mbf{I} & \mbf{0} & \mbf{0} \\
2\mbf{I} & 2\mbf{C} +6\mbf{I} & \mbf{0} \\
\end{bmatrix}^{\mbf{+}}\begin{bmatrix}
\mbf{b}_{1} \\
\mbf{b}_{2} \\
\mbf{b}_{3} \\
\end{bmatrix}\\&=\begin{bmatrix}
\mbf{0} & \frac{1}{2}( 2\mbf{C} +6\mbf{I})^{-1} & \mbf{0} \\
\mbf{0} & - ( 2\mbf{C}+6\mbf{I})^{-2} & ( 2\mbf{C} +6\mbf{I})^{-1} \\
\mbf{0} & \mbf{0} & \mbf{0} \\
\end{bmatrix}\begin{bmatrix}
\mbf{b}_{1} \\
\mbf{b}_{2} \\
\mbf{b}_{3} \\
\end{bmatrix}\\&=\begin{bmatrix}
\frac{1}{2}( 2\mbf{C} +6\mbf{I})^{-1}\mbf{b}_{2} \\
-( 2\mbf{C} +6\mbf{I})^{-2}\mbf{b}_{2} \\
\mbf{b}^{*} \\
\end{bmatrix}\\
\endaligned
\end{equation}

where
 \(\mbf{b}_{1}=\begin{bmatrix} 0 \\ 0 \\ \end{bmatrix}\),
\(\mbf{b}_{2}=\begin{bmatrix} 2 \\  - 4 \\ \end{bmatrix}\),
\(\mbf{b}_{3}=\begin{bmatrix} 0 \\ 0 \\ \end{bmatrix},\)
\(\mbf{b}^{*}=\begin{bmatrix} 0 \\ 0 \\ \end{bmatrix}\).\\

Let us verify that the product of the following block matrices is the identity matrix
\[\aligned &\begin{bmatrix}
4\mbf{C} +12\mbf{I} & \mbf{0} \\
2\mbf{I} & 2\mbf{C} +6\mbf{I} \\
\end{bmatrix}\begin{bmatrix}
\dfrac{1}{2}(2\mbf{C} +6\mbf{I})^{-1} & \mbf{0} \\
-( 2\mbf{C}+6\mbf{I})^{- 2} & (2\mbf{C} +6\mbf{I})^{- 1} \\
\end{bmatrix} \\&= \begin{bmatrix}
(4\mbf{C} +12\mbf{I})\dfrac{1}{2}( 2\mbf{C}+6\mbf{I})^{-1} & \mbf{0} \\
2\mbf{I}\dfrac{1}{2}( 2\mbf{C}+6\mbf{I})^{-1} - (2\mbf{C} +6\mbf{I}
( 2\mbf{C} +6\mbf{I})^{- 2} & (2\mbf{C}+6\mbf{I})(2\mbf{C}
+6\mbf{I})^{- 1} \end{bmatrix} \\&=
\begin{bmatrix}
\mbf{I} & \mbf{0} \\
( 2\mbf{C}+6\mbf{I})^{- 1} - (2\mbf{C} +6\mbf{I})^{-1} & \mbf{I} \\
\end{bmatrix} = \begin{bmatrix}
\mbf{I} & \mbf{0} \\
\mbf{0} & \mbf{I} \\
\end{bmatrix}
\endaligned\]
The same is true in reverse order of multiplying the matrices. 
Let us compute the elements of the last matrix in \eqref{aac}

\[\frac{1}{2}( 2\mbf{C }+6\mbf{I})^{-1}\mbf{b}_{2}=\dfrac{1}{2}
\begin{bmatrix}
0 & - 4 \\
4 & 0 \\
\end{bmatrix}^{- 1}\begin{bmatrix}
2 \\
 - 4 \\
\end{bmatrix} = \dfrac{1}{2}\begin{bmatrix}
0 & \dfrac{1}{4} \\
 - \dfrac{1}{4} & 0 \\
\end{bmatrix}\begin{bmatrix}
2 \\
 - 4 \\
\end{bmatrix} = \begin{bmatrix}
 - \dfrac{1}{2} \\
 - \dfrac{1}{4} \\
\end{bmatrix}\]

\[\aligned
-( 2\mbf{C} +6\mbf{I})^{-2}\mbf{b}_{2}&= -\begin{bmatrix}
0 & - 4 \\
4 & 0 \\
\end{bmatrix}^{-2}\begin{bmatrix}
2 \\
 - 4 \\
\end{bmatrix}= -\begin{bmatrix}
0 & \dfrac{1}{4} \\
 - \dfrac{1}{4} & 0 \\
\end{bmatrix}^{2}\begin{bmatrix}
2 \\
 - 4 \\
\end{bmatrix}\\& = - \begin{bmatrix}
 - \dfrac{1}{16} & 0 \\
0 & - \dfrac{1}{16} \\
\end{bmatrix}\begin{bmatrix}
2 \\
 - 4 \\
\end{bmatrix} = \begin{bmatrix}
\dfrac{1}{8} \\
 - \dfrac{1}{4} \\
\end{bmatrix}
\endaligned\]

\[\mbf{0}\mbf{b}_{1}+\mbf{0}\mbf{b}_{2}+\mbf{0}\mbf{b}_{3}=\begin{bmatrix}
0 \\
0 \\
\end{bmatrix}\]

From this we can find the  particular solution

\[\mbf{y}_{B} = \begin{bmatrix}
\dfrac{1}{2}( 2\mbf{C} +6\mbf{I})^{-1}\mbf{b}_{2} \\
-( 2\mbf{C} +6\mbf{I})^{-2}\mbf{b}_{2} \\
\mbf{b}^{*} \\
\end{bmatrix}=
\renewcommand{\arraystretch}{2.1}\left[
\begin{array}{c}
 - \dfrac{1}{2} \\
 - \dfrac{1}{4} \\\hline
\dfrac{1}{8} \\
 - \dfrac{1}{4} \\\hline
0 \\
0 \\
\end{array}\right]
=\begin{bmatrix}
 - \dfrac{1}{2} \\
 - \dfrac{1}{4} \\
\dfrac{1}{8} \\
 - \dfrac{1}{4} \\
0 \\
0 \\
\end{bmatrix}\]
or

\[y = - \frac{1}{2}x^{2}e^{-3x}\sin{2x} - \frac{1}{4}\ x^{2}e^{-3x}\cos{2x} + \frac{1}{8} xe^{- 3x}\sin{2x} - \frac{1}{4}xe^{- 3x}\cos{2x}\]

Although the original matrix \(\mathcal{D}_B\) was $6\times 6$, during the
calculation of particular solution of the differential equation \eqref{rblc}
we calculated the $2\times 2$ inverse matrix only. 
\qed

\section{Conclusion}

The operational method is a fast and universal mathematical tool for
obtaining solutions of differential equations.  Determining particular solutions of ordinary nonhomogeneous linear differential equations with constant coefficients using the undetermined coefficients method and the differential operator method are generally known.In particular, distinct from the differential operator method introduced in the literature, we propose and highlight utilizing the definition of the pseudoinverse matrix differential operator to determine a
particular solution of differential equations. This method is simple to understand and to apply, compared to some cases of the differential operator method. However, it requires the determination of an inverse or pseudoinverse matrix, which generally has a special type. If the matrix is
singular, then we will take a pseudoinverse matrix - Moore Penrose pseudoinverse matrix - instead of the inverse matrix. The paper shows that for its determination, we need to calculate only the inverse submatrix of the considered matrix. Finally, the technique of calculating a particular solution by expressing a matrix differential operator as a block matrix is illustrated. Combination of the operational method, the method of undetermined coefficients and application of the matrix differential method provides a powerful instrument to determine particular solutions of differential equations.

\end{document}